\newif\ifTR

\TRtrue

\ifTR

\documentclass[a4paper,11pt]{article}
\usepackage{authblk}
\usepackage{natbib}
\usepackage[cm]{fullpage}
\usepackage{caption}
\usepackage{subcaption}
\usepackage{amsthm}
\usepackage{amsmath}
\usepackage{amssymb}
\usepackage{adjustbox}

\allowdisplaybreaks[0]
\newtheorem{definition}{Definition}

\newcommand{\TABLE}[3]{%
	\centering
	\caption{#1}%
	\adjustbox{max width=\linewidth}{#2}%
	\ifx\relax#3\relax\else
	\par\smallskip\noindent\footnotesize #3%
	\fi
}
\newcommand{\FIGURE}[3]{%
	\centering
	#1
	\caption{#2}
	\ifx\relax#3\relax\else
	\par\smallskip\noindent\footnotesize #3
	\fi
}

\newcommand{\ACKNOWLEDGMENT}[1]{%
	\section*{Acknowledgments}%
	#1%
}

\newcommand{\ECRRHFirstLine}[1]{}
\newcommand{\ECLRHFirstLine}[1]{}
\newcommand{\ECRRHSecondLine}[1]{}
\newcommand{\ECLRHSecondLine}[1]{}

\newcommand{\ECSwitch}{%
	\newpage
	\setcounter{section}{0}%
	\renewcommand{\thesection}{EC.\arabic{section}}%
	\renewcommand{\theequation}{EC.\arabic{equation}}%
	\renewcommand{\thetable}{EC.\arabic{table}}%
	\renewcommand{\thefigure}{EC.\arabic{figure}}%
}

\newcommand{\ECHead}[1]{%
	\section*{#1}%
}

\newcommand{\ABSTRACT}[1]{%
	\begin{abstract}
		#1
	\end{abstract}
}

\newcommand{\KEYWORDS}[1]{%
	\par\medskip\noindent\textbf{Keywords:} #1\par
}

\renewenvironment{proof}[1]{%
	\par\noindent\textit{Proof.}\enspace
}{%
	\hfill$\square$\par\medskip
}

\newcommand{\Halmos}{} 

\newenvironment{APPENDICES}
{%
    \newpage
    \appendix
    \setcounter{section}{0}%
    \crefname{section}{Appendix}{Appendices}
    \Crefname{section}{Appendix}{Appendices}
}
{%
}
\usepackage[capitalize,noabbrev]{cleveref}
\crefname{appsec}{Appendix}{Appendices}
\Crefname{appsec}{Appendix}{Appendices}

\else

\documentclass[opre,dblanonrev]{informs4}
\RequirePackage{tgtermes}
\RequirePackage{newtxtext}
\RequirePackage{newtxmath}
\RequirePackage{bm}
\RequirePackage{endnotes}

\OneAndAHalfSpacedXI

\EquationsNumberedThrough    

\TheoremsNumberedThrough     
\ECRepeatTheorems  %

\MANUSCRIPTNO{MOOR-0001-2024.00}

\fi

\usepackage{algorithm}
\usepackage{algpseudocode}
\usepackage{tikz}

\usepackage{natbib}
 \bibpunct[, ]{(}{)}{,}{a}{}{,}%
\usepackage{paralist}
\usepackage[dvipsnames]{xcolor}
\usepackage{comment}

\usepackage[capitalize,noabbrev]{cleveref}

\usepackage{booktabs}
\usepackage{multirow}
\usepackage{multicol}

\usepackage{xspace}
\newcommand{\probabbr}{CVRPTW-TD\xspace}   
\newcommand{\Veh}{K} 
\newcommand{\VehCap}{Q}
\newcommand{\Tmax}{T_{\mathrm{max}}}
\newcommand{\TDVis}{V_D} 

\newcommand{\dmin}[1]{\delta^{\mathrm{min}}_{#1}} 
\newcommand{\dmax}[1]{\delta^{\mathrm{max}}_{#1}}
\newcommand{\dminAll}{\dmin{}}
\newcommand{\dmaxAll}{\dmax{}}

\newcommand{\AllFrag}{\mathcal{F}}

\newcommand{\esF}[2]{t_{#2}^{\mathrm{ear}}}
\newcommand{\lsF}[2]{t_{#2}^{\mathrm{lat}}}
\newcommand{\schedF}[1]{\mathcal{S}_{#1}}
\newcommand{\lenF}[1]{{|#1|}}

\newcommand{\durF}[1]{\mathrm{dur}_{#1}}
\newcommand{\demF}[1]{q_{#1}}
\newcommand{\costF}[1]{c_{#1}}

\newcommand{\FragFromTo}[2]{\AllFrag_{#1 \to #2}}
\newcommand{\FragFrom}[1]{\AllFrag_{#1 \to *}}
\newcommand{\FragTo}[1]{\AllFrag_{* \to #1}}

\newcommand{\FSECVehMin}[1]{V_{\mathrm{min}}(#1)}

\NewDocumentCommand{\FORM}{m o}{%
    \IfNoValueTF{#2}
      {\ensuremath{\mathrm{#1}}\xspace}%
      {\ensuremath{\mathrm{#1}(#2)}\xspace}%
}

\NewDocumentCommand{\FF}{o}{%
    \IfNoValueTF{#1}
      {\FORM{FF}}%
      {\FORM{FF}(#1)}%
}

\NewDocumentCommand{\FFVI}{o}{%
    \IfNoValueTF{#1}
      {\FORM{FF^+}}%
      {\FORM{FF^+}(#1)}%
}

\NewDocumentCommand{\FFVIm}{o}{%
    \IfNoValueTF{#1}
      {\FORM{FF^*}}%
      {\FORM{FF^*}(#1)}%
}

\NewDocumentCommand{\relFF}{o}{%
    \IfNoValueTF{#1}
      {\FORM{\widetilde{FF}}}%
      {\FORM{\widetilde{FF}}(#1)}%
}

\NewDocumentCommand{\relFFVI}{o}{%
    \IfNoValueTF{#1}
      {\FORM{\widetilde{FF}^+}}%
      {\FORM{\widetilde{FF}^+}(#1)}%
}

\NewDocumentCommand{\relFFVIm}{o}{%
    \IfNoValueTF{#1}
      {\FORM{\widetilde{FF}^*}}%
      {\FORM{\widetilde{FF}^*}(#1)}%
}

\NewDocumentCommand{\REL}{m}{%
    {\ensuremath{\widetilde{#1}}\xspace}
}

\NewDocumentCommand{\VAL}{o}{%
    {\ensuremath{z(#1)}\xspace}
}

\newcommand{\solalg}{\ensuremath{\mathrm{FSA}}\xspace}   

\newcommand{\CGfrag}{\AllFrag^{CG}}
\newcommand{\CGfragFromTo}[2]{\AllFrag^{CG}_{#1 \to #2}}

\newcommand{\CGtimeMIP}{t_\mathrm{guess}}

\newcommand{\ENfrag}{\AllFrag^{E}}
\newcommand{\ENfragFrom}[1]{\AllFrag^{E}_{#1 \to *}}
\newcommand{\ENfragTo}[1]{\AllFrag^{E}_{* \to #1}}

\newcommand{\maxfragenum}{F_\mathrm{max}}

\newcommand{\LBalg}{LB_\mathrm{sol}}
\newcommand{\UBguess}{UB_\mathrm{cand}}
\newcommand{\UBstep}{gap_\mathrm{step}}
\newcommand{\UBinitstep}{gap_\mathrm{init}}
\newcommand{\IncSolVal}{UB_\mathrm{sol}}
\newcommand{\IncSol}{\mathcal{F}_\mathrm{sol}}

\newcommand{\solalgNoBC}{\ensuremath{\mathrm{FSA^{NC}}}\xspace}   
\newcommand{\solalgNoBCNoCut}[1]{\ensuremath{\mathrm{FSA^{NC}-#1}}\xspace}

\newcommand{\reltempdep}{\ensuremath{\sigma}}

\newcommand{\dMaxVeh}{\gamma}
\newcommand{\dAss}[1]{\zeta_{#1}}
\newcommand{\dFlow}[1]{\eta_{#1}}

\newcommand{\dTimeMTZ}[1]{\theta_{#1}}
\newcommand{\dTimeLB}[1]{\kappa_{#1}}
\newcommand{\dTimeUB}[1]{\lambda_{#1}}

\newcommand{\dLoadMTZ}[1]{\tau_{#1}}
\newcommand{\dLoadLB}[1]{\nu_{#1}}
\newcommand{\dLoadUB}[1]{\xi_{#1}}

\newcommand{\rcArc}[2]{\tilde{c}_{#1 #2}} 
\newcommand{\rcComplFrag}[1]{\pi^{#1}}

\newcommand{\flCost}[1]{\tilde{c}^{#1}}
\newcommand{\flSet}[1]{N^{#1}}
\newcommand{\flEnd}[1]{e^{#1}}
\newcommand{\flStart}[1]{s^{#1}}
\newcommand{\flDur}[1]{\mathrm{dur}^{#1}}
\newcommand{\flLS}[1]{t_{\mathrm{lat}}^{#1}}
\newcommand{\flES}[1]{t_{\mathrm{ear}}^{#1}}
\newcommand{\flLoad}[1]{q^{#1}}
\newcommand{\flLSDMin}[1]{t_{\dminAll}^{#1}}

\newcommand{\rcBoundDom}[2]{\rho(#1, #2)}

\newcommand{\xsol}{\bar{x}}

\newcommand{\esT}[1]{t_{#1}^{\mathrm{es}}}
\newcommand{\lsT}[1]{t_{#1}^{\mathrm{ls}}}
\newcommand{\eaT}[1]{t_{#1}^{\mathrm{ea}}}
\newcommand{\ldT}[1]{t_{#1}^{\mathrm{ld}}}

\newcommand{\dohnalg}{\ensuremath{\mathrm{DOHN}}\xspace}
\newcommand{\mtzalg}{\ensuremath{\mathrm{AF}}\xspace}

\newtheorem{observation}{Observation}

\usepackage{thmtools}

\usepackage{tikz}
\usepackage{pgfplots}
\usepackage{pgfkeys}
\usepackage{pgfplotstable} 
\usepgfplotslibrary{groupplots}
\pgfplotsset{compat=newest}
\usepgfplotslibrary{external} 					

\usetikzlibrary{backgrounds}					
\usetikzlibrary{fit} 							
\pgfplotsset{major grid style={thick}}
\usetikzlibrary{arrows}
\usetikzlibrary{arrows.meta}
\usetikzlibrary{shapes.geometric}
\usetikzlibrary{plotmarks}

\usetikzlibrary{arrows,positioning,shapes.geometric,calc,decorations}
\usepackage{ifthen}
\usepackage[normalem]{ulem}

\tikzstyle{black}=[draw=black,circle,fill=white,minimum size=5pt, inner sep=1pt]
\tikzstyle{gray}=[draw=gray,circle,fill=white,minimum size=5pt, inner sep=1pt]

\usetikzlibrary{patterns}

\newcommand{\includeNodesExample}{
\node[depot] (n0) at (52,35) {0};

\node[customer_ss] (n1) at (23,35.5) {1};
\node[left=2mm of n1] {[1,8]};

\node[customer] (n2) at (40,30) {2};
\node[above=0mm of n2] {[2,6]};

\node[customer] (n3) at (71.3,25) {3};
\node[right=1mm of n3] {[2,9]};

\node[customer_pmin] (n4) at (67,46) {4};
\node[below=1mm of n4] {[1,6]};

\node[customer] (n5) at (35,38) {5};
\node[above=0mm of n5] {[1,8]};

\node[customer] (n6) at (30,50) {6};
\node[left=1mm of n6] {[2,8]};

\node[customer] (n7) at (42,22) {7};
\node[above=0.0mm of n7] {[6,7]};

\node[customer] (n8) at (84,50) {8};
\node[right=1mm of n8] {[6,9]};

\node[customer_pmax] (n9) at (80,41.5) {9}; 
\node[right=1.0mm of n9] {[2,5]};

\node[customer] (n10) at (83,30) {10};
\node[right=1mm of n10] {[5,6]};

\node[customer] (n11) at (45,50) {11};
\node[below left=0mm of n11] {[1,9]};

\node[customer] (n12) at (62, 32) {12};
\node[right=1mm of n12] {[1,6]};

\node[customer_ss] (n13) at (23,31) {13}; 
\node[left=2mm of n13] {[1,8]};

\node[customer] (n14) at (52,26) {14}; 
\node[right=1mm of n14] {[8,9]};

\node[customer_pmin] (n15) at (64,49.2) {15};
\node[left=1mm of n15] {[3,9]};

\node[customer_pmax] (n16) at (80,37) {16};
\node[right=1mm of n16] {[6,9]};

\node[draw, dotted, very thick, ellipse, fit=(n1)(n13), inner sep=1pt] {};
\node[draw, dotted, very thick, ellipse, fit=(n4)(n15), inner sep=1pt] {};
\node[draw, dotted, very thick, ellipse, fit=(n9)(n16), inner sep=1pt] {};
}

\newcommand{\includeSynNodesExample}{
\node[depot] (n0) at (52,35) {0};

\node[customer_ss] (n1) at (23,35.5) {1};
\node[left=2mm of n1] {[1,8]};

\node[customer_pmin] (n4) at (67,46) {4};
\node[below=1mm of n4] {[1,6]};

\node[customer_pmax] (n9) at (80,41.5) {9}; 
\node[right=1.0mm of n9] {[2,5]};

\node[customer_ss] (n13) at (23,31) {13}; 
\node[left=2mm of n13] {[1,8]};

\node[customer_pmin] (n15) at (64,49.2) {15};
\node[left=1mm of n15] {[3,9]};

\node[customer_pmax] (n16) at (80,37) {16};
\node[right=1mm of n16] {[6,9]};

\node[draw, dotted, very thick, ellipse, fit=(n1)(n13), inner sep=1pt] {};
\node[draw, dotted, very thick, ellipse, fit=(n4)(n15), inner sep=1pt] {};
\node[draw, dotted, very thick, ellipse, fit=(n9)(n16), inner sep=1pt] {};
}

\newcommand{\includeStyleExample}{
\tikzstyle{depot} = [rectangle, draw=black, fill=red!60, minimum size=5mm, inner sep=0pt, text centered]
\tikzstyle{customer} = [circle, draw=black, fill=blue!30, minimum size=5mm, inner sep=0pt, text centered]
\tikzstyle{customer_ss} = [circle, draw=black, fill=yellow!50, minimum size=5mm, inner sep=0pt, text centered]
\tikzstyle{customer_d} = [circle, draw=black, fill=orange!50, minimum size=5mm, inner sep=0pt, text centered]
\tikzstyle{customer_pmax} = [circle, draw=black, fill=green!50, minimum size=5mm, inner sep=0pt, text centered]
\tikzstyle{customer_pmin} = [circle, draw=black, fill=purple!50, minimum size=5mm, inner sep=0pt, text centered]

\tikzstyle{route} = [very thick, ->, >=Stealth]

\tikzstyle{fragNodeToD} = [very thick, densely dotted, ->, >=Stealth]
\tikzstyle{fragDToNode} = [very thick, loosely dashed, ->, >=Stealth]
\tikzstyle{fragNodeToNode} = [very thick, dashdotted, ->, >=Stealth]
\tikzstyle{fragDToD} = [very thick, loosely dashdotdotted, ->, >=Stealth]
}

\newcommand{\drawRouteOne}{
\draw[route, red] (n0) -- (n12);
\draw[route, red] (n12) -- (n3);
\draw[route, red] (n3) -- (n10);
\draw[route, red] (n10) -- (n16);
\draw[route, red] (n16) -- (n0);
}

\newcommand{\drawRouteOneFragForm}{
\draw[fragDToNode, red] (n0) to[bend right=8] (n16);
\draw[fragNodeToNode, red] (n16) to[bend right=8] (n0);
}

\newcommand{\drawRouteTwo}{
\draw[route, blue] (n0) -- (n5);
\draw[route, blue] (n5) -- (n1);
\draw[route, blue] (n1) -- (n6);
\draw[route, blue] (n6) -- (n11);
\draw[route, blue] (n11) -- (n0);
}

\newcommand{\drawRouteTwoFragForm}{
\draw[fragDToNode, blue] (n0) to[bend left=2] (n1);
\draw[fragNodeToD, blue] (n1) to[bend left=10] (n0);
}

\newcommand{\drawRouteThree}{
\draw[route, OliveGreen] (n0) -- (n2);
\draw[route, OliveGreen] (n2) -- (n13);
\draw[route, OliveGreen] (n13) -- (n7);
\draw[route, OliveGreen] (n7) -- (n14); 
\draw[route, OliveGreen] (n14) -- (n0);
}

\newcommand{\drawRouteThreeFragForm}{
\draw[fragDToNode, OliveGreen] (n0) to[bend right=2] (n13);
\draw[fragNodeToD, OliveGreen] (n13) to[bend right=10] (n0);
}

\newcommand{\drawRouteFour}{
\draw[route, brown] (n0) -- (n4);
\draw[route, brown] (n4) -- (n9);
\draw[route, brown] (n9) -- (n8);
\draw[route, brown] (n8) -- (n15);
\draw[route, brown] (n15) -- (n0);
}

\newcommand{\drawRouteFourFragForm}{
\draw[fragDToNode, brown] (n0) to[bend right=3] (n4);
\draw[fragNodeToNode, brown] (n4) to[bend left=5] (n9);
\draw[fragNodeToNode, brown] (n9) to[bend right=40, looseness=1] (n15);
\draw[fragNodeToD, brown] (n15) to[bend right=3] (n0);
}

\ifTR

\title{An exact algorithm for vehicle routing problems with temporal dependency constraints}
\author{Loek van Montfort, Markus Leitner, and Rosario Paradiso}
\affil{Department of Operations Analytics, Vrije Universiteit Amsterdam, The Netherlands.
\texttt{\{l.g.a.j.van.montfort2|m.leitner|r.paradiso\}@vu.nl}}

\fi

\begin{document}



\ifTR

\maketitle

\fi 

\ABSTRACT{%
	Temporal dependencies between customer visits, such as synchronization constraints, pose a fundamental challenge in vehicle routing.
	These dependencies, which arise in applications such as home healthcare routing, aircraft scheduling, and technician routing, introduce inter-route constraints that make the resulting problems significantly harder to solve.
	We present an exact solution method for vehicle routing problems with temporal dependencies capable of handling all types of temporal dependencies studied in the literature, unlike most existing approaches that target specific subclasses.
	Our approach is based on a fragment-based formulation in which routes are represented as sequences of a new type of fragment, designed to handle temporal dependency constraints.
	This formulation is solved via a price-cut-and-enumerate algorithm that computes a lower bound using alternating column-and-row generation, obtains an initial upper bound, and iteratively refines both bounds through fragment enumeration and branch-and-cut, supported by several new classes of valid inequalities.
	Computational experiments show that our method significantly outperforms state-of-the-art benchmark methods and is able to solve previously intractable instances while covering a wider range of temporal dependencies.
}%




\KEYWORDS{Vehicle Routing Problem, Temporal Dependencies, Synchronization, Price-Cut-and-Enumerate, Exact methods} 




\section{Introduction}\label{sec:Intro}
Vehicle Routing Problems (VRPs), which address the optimal design of vehicle routes for a given set of tasks, have been intensively studied since their introduction by \citet{Dantzig59}. Important variants include the capacitated VRP or the VRP with time windows \citep{Toth2014}.
A common assumption in most problem variants is that whether and when a task can be performed is independent of other vehicles' routes and schedules. While valid in many applications, this assumption does not hold for VRPs requiring some form of synchronization between tasks or vehicles, which have recently attracted growing interest \citep{SOARES2024817,Drexl2012}. The survey by \citet{SOARES2024817} distinguishes between operation synchronization, induced by temporal dependencies between tasks, and movement synchronization, in which the movement of one vehicle requires a corresponding movement of another.
This article focuses on operation synchronization arising from \textit{temporal dependencies} between pairs of tasks, which may be performed by the same or different vehicles and can therefore appear as intra-route and inter-route constraints. These dependencies are classified into the following types:
\begin{inparaenum}[(i)]
	\item \textit{synchronization} requires two tasks to start simultaneously;  
	\item \textit{precedence} imposes a predetermined starting order of two tasks and can be generalized by specifying 
	\item a \textit{minimum or maximum difference} between their start times;
	\item \textit{overlap} requires that the execution intervals of two tasks intersect, while 
	\item \textit{non-overlap} prohibits their simultaneous execution. 
\end{inparaenum}

The relevance of temporal dependencies in routing and scheduling problems has been shown for a broad range of applications. In home healthcare (HHC), synchronization arises when patients require multiple caregivers simultaneously \citep{BREDSTROM200819}, while precedence or general temporal constraints apply when medication is needed a certain time before or after dinner \citep{Mankowska2014}. Non-overlap constraints enable the integration of different home care structures \citep{Frifita2020} or the splitting of multi-task visits across different caregivers \citep{VANMONTFORT2025104235}. 
Further applications include the delivery and 
installation of large items \citep{HOJABRI201887}, collaboration between teams 
at customer locations \citep{HA2020105085}, airport ground handling when jobs require multiple workers \citep{DOHN20091145}, air fleet planning to ensure 
similar starting times on different days 
\citep{IOACHIM199975}, therapist scheduling when patients require multiple therapists 
\citep{Kling2025}, electric vehicle routing to prevent 
queuing at charging stations
\citep{Froger2022}, and
last-mile delivery by 
hierarchically ordered vehicles such as vans or bikes in combination with self-driving robots \citep{CHEN20211164,Zhao2024}.
Temporal dependencies can also increase efficiency by enabling individual routes and schedules rather than fixing predetermined teams that must travel together, see, e.g., \cite{DEAGUIAR2023101503,HANAFI2020515}.

Although routing problems with temporal dependencies have received considerable attention, most exact solution methods focus on a single type, such as synchronization \citep{DOHN20091145,Luo2016, Doulabi2020, Qiu2022} or maximum time differences \citep{Bianchessi2019}.
Only a few can handle multiple types, including synchronization, overlap, and minimum/maximum time differences \citep{Dohn2011, RASMUSSEN2012598, TILK2019549}, and none is able to handle non-overlap constraints.


A key reason for the scarcity of exact solution methods is that state-of-the-art algorithms for classical VRPs, typically based on route-based formulations, cannot be extended easily 
to VRPs with temporal dependencies. Temporal restrictions complicate or invalidate dominance rules that enable the elimination of suboptimal routes, which are central to the effectiveness of column-generation-based VRP algorithms. 
As noted by \cite{GSCHWIND201560} and \cite{TILK2019549}, this occurs because temporal dependencies introduce a trade-off for the starting time of tasks within a route. Instead of starting tasks as early as possible, it may be beneficial to delay a task to satisfy temporal dependencies within or between routes. 

Two main approaches have been proposed to address this challenge: 
\begin{inparaenum}[(i)]
\item relax temporal dependencies in the master and subproblem of branch-and-price (BP) approaches and enforce them via time window branching \citep{Bredstrom2007, Dohn2011}, or
\item incorporate temporal dependencies in the master problem and develop dedicated algorithms and dominance rules for the resulting, more challenging pricing subproblem, which typically involves linear node costs \citep{IOACHIM199975,TILK2019549}.
\end{inparaenum}
These two strategies represent contrasting approaches to deal with the complexity imposed by temporal dependencies. Time window branching based methods benefit from a well-studied pricing subproblem for which efficient methods exist, but yield weaker linear relaxation bounds since they may generate infeasible routes that ignore information about temporal dependencies. Incorporating temporal dependencies in the master problem produces tighter formulations, but comes at the cost of a significantly harder subproblem due to linear node costs \citep{IOACHIM199975} whose difficulty is further increased when temporal dependencies can arise between tasks of a single route \citep{GSCHWIND201560}.

In this paper, we study a generic VRP with temporal dependencies (VRPTD) and propose an exact solution algorithm applicable to a broad class of temporal dependencies, including non-overlapping constraints not addressed by any existing exact method. 
Our approach is based on a novel, fragment-based formulation designed to achieve a computationally beneficial trade-off between formulation strength and pricing problem complexity, instead of prioritizing one at the expense of the other as in existing methods. Specifically, fragments are constructed such that no task included in such a route part can have conflicting scheduling incentives due to inter- or intra-route constraints. While representing routes by sequences of partial routes (fragments) has been used in other VRPs \citep{Paradiso2020,Rist2021,Yang2023}, we are, to the best of our knowledge, the first to introduce and define fragments tailored to routing problems with temporal dependencies between tasks. 
The main contributions of this paper are:
\begin{itemize}
    \item We introduce 
    a generic VRPTD that, as opposed to existing formulations, includes non-overlap constraints between visits that may belong to
    the same or different routes. 
    \item We demonstrate the benefits of a fragment-based route representation for routing problems with temporal dependencies and introduce a novel fragment-based formulation for the VRPTD.
    \item We propose four classes of strengthening inequalities for the new formulation.
    \item We develop an exact price-cut-and-enumerate 
    algorithm that first uses row-and-column generation to obtain a lower bound and a mixed-integer linear programming (MILP)-based heuristic for an upper bound. Subsequently, it refines these bounds through fragment enumeration, fragment reduction, and branch-and-cut in an iterative process. 
    \item We perform an extensive computational study demonstrating that the proposed framework handles a wider range of temporal dependencies than existing methods and outperforms two benchmark approaches. The results also confirm the effectiveness of the proposed strengthening inequalities. 
\end{itemize}

The remainder of this paper is organized as follows. \cref{sec:litrev} reviews exact solution methods for VRPs with temporal dependencies. \cref{sec:probdef} defines the problem and presents the fragment-based formulation. \cref{sec:valineq} introduces the strengthening inequalities, and \cref{sec:solappr} details the price-cut-and-enumerate algorithm. \cref{sec:compres} reports the computational results, and \cref{sec:Conclusion} summarizes the findings and outlines directions for future research. Proofs of theoretical results are given in the appendix which also provides tables summarizing notation, details of implementation aspects, and additional results.

\section{Literature review}\label{sec:litrev}

We review exact solution methods for VRPs with operation synchronization induced by temporal dependencies. Broader overviews that include movement synchronization and heuristic methods can be found in the surveys by \cite{SOARES2024817} and \cite{Drexl2012}. 
We first discuss BP algorithms that ensure temporal dependencies in the master problem or via time window branching, before briefly reviewing related exact methods not based on column generation.

\paragraph{Branch-and-price methods ensuring temporal dependencies in the master problem}

\cite{IOACHIM199975} were the first to incorporate temporal dependency constraints explicitly in the master problem of a BP approach. They study an aircraft fleet routing and scheduling problem in which identical flights on different days must depart at the same time. 
Their set-partitioning formulation models synchronization via the starting times of temporal dependency tasks within each route. 
The resulting subproblem is an elementary shortest path problem with resource constraints and linear time node costs, which is already strongly NP-hard without the latter costs \citep{Spliet23}. They solve this subproblem using the dynamic programming algorithm of \cite{Ioachim1998}, which uses a piecewise linear function for partial paths to handle the linear node costs.

\cite{TILK2019549} study a variant of a VRPTW with soft demand, minimizing a combination of dispatching costs, travel costs, and collected profits. Temporal dependencies are specified by minimum and maximum starting time differences between task pairs, an approach that can also accommodate synchronization and overlap. The interdependence of three resources (cost, load, and time) leads to two trade-offs for (partial) routes (cost vs.\ load and cost vs.\ time). As this 
 complicates the pricing subproblem and only allows for weak dominance rules between partial routes, they do not use dynamic programming for solving the pricing suproblem of their route-based formulation. 
Instead, they propose a nested decomposition approach in which the subproblem is solved by branch-price-and-cut (BPC).

\cite{Qiu2022} study a VRPTW arising in HHC with caregiver qualifications and task pairs
that must be performed within a maximum time. Each caregiver can perform at most one task with temporal dependencies, while at most two such tasks can exist per patient. They propose a BPC approach based on a set-partitioning formulation that enforces synchronization via inequalities in the master problem and develop a dedicated labeling algorithm for the pricing subproblem.

\paragraph{Branch-and-price methods ensuring temporal dependencies via time window branching}

\cite{Bredstrom2007} were the first to propose a BP approach for a VRPTW with synchronization in which temporal dependencies are relaxed in both the master and subproblem and enforced via time window branching. While relaxing temporal dependencies weakens the linear programming relaxation, it enables solving the pricing problem by (slightly adjusted versions of) standard labeling algorithms successfully used for VRPs without temporal dependencies.

Time window branching in a BP approach has also been used by \cite{DOHN20091145} for a manpower allocation problem in which tasks require multiple qualified teams simultaneously and, more recently, by \cite{Kling2025} for the planning of physical therapy within a hospital, where a subset of visits requires two therapists.
\cite{Luo2016} study a manpower allocation problem with tasks that require specific numbers of workers from given categories. 
Their BPC approach uses a combination of time window branching and master-level constraints (infeasible path and weak clique inequalities) to ensure synchronization. 

\cite{Dohn2011, RASMUSSEN2012598} and \cite{LIN2021102177} propose BP(C) algorithms based on time window branching that can handle multiple types of temporal dependencies. 
Similar to \cite{TILK2019549}, they use a generic framework to specify these 
dependencies and demonstrate the applicability of time window branching to synchronization, minimum and/or maximum time differences, precedence, and overlap. To our knowledge, the work of \cite{Dohn2011} is the only one comparing a BP algorithm with time window branching to a BPC ensuring temporal dependencies by dynamically adding inequalities to the master problem. Both approaches show a comparable performance in their computational study.

\paragraph{Other exact approaches}

In the following, we discuss additional decomposition-based methods, thereby excluding methods statically incorporating temporal dependencies in the master problem. Thus, we will not review approaches based on solving (compact) MILP formulations using a general-purpose solver \citep[see, e.g.,][]{Algendi2025,Kergosien2009,Lopez2018,VANMONTFORT2025104235}
and cutting-plane based methods ensuring temporal dependencies via variables and constraints statically included in the master problem whose cuts focus on other aspects such as the routing aspect to strengthen the dual bounds, see, e.g., \cite{Cappanera2020,Chabot02112017}.
%
\cite{Doulabi2020} study a VRPTW with synchronization and stochastic travel and service times in which customers require the simultaneous service of multiple vehicles. 
They propose a two-stage MILP which is solved via a branch-and-cut implementation of the L-shaped algorithm, dynamically adding subtour elimination, service time capacity, and no-overlap constraints.
%
\cite{Mankowska2016} propose a Benders decomposition approach for HHC routing and scheduling with synchronization and minimum/maximum time differences between visits, where routing decisions are made in the master problem and starting times are determined in the subproblem. 
\cite{HANAFI2020515} study a team orienteering problem in which customer can require multiple visits with predefined precedence relations. They propose a branch-and-cut algorithm that dynamically separates inequalities related to precedence restrictions in addition to connectivity constraints. 

\smallskip

Overall, the two most frequently used methods for VRPs with 
temporal dependencies are complementary BP approaches based on route formulations that either incorporate the temporal restriction in the master problem \citep[see, e.g.,][]{IOACHIM199975,TILK2019549} or relax them in the master and subproblem and use time window branching \citep[see, e.g.,][]{DOHN20091145,RASMUSSEN2012598,LIN2021102177}. The former maintains strong dual bounds at the cost of a computationally more challenging subproblem; the latter sacrifices the quality of the dual bounds for a simpler subproblem. 
Notably, \cite{GSCHWIND201560} show that incorporating temporal dependencies arising only between customers in the same route in the pricing problem strengthens the formulation at the cost of added complexity, but yields a better performing algorithm.

The approach proposed in this article is able to handle general temporal dependencies including non-overlapping constraints which are not considered by any exact method. Furthermore, it achieves a favorable trade-off between formulation strength and pricing subproblem complexity.

\section{Problem definition and formulation}\label{sec:probdef}
The Capacitated VRP with Time Windows and Temporal Dependencies (\probabbr) considers a set of homogeneous vehicles $\Veh$ with capacity $\VehCap$ that are initially located at the depot $0$ and must perform all tasks in the set $V = \{1,\dots,n\}$ within planning horizon $[0,\Tmax]$. Parameter $t_{uv}$ indicates the time required for a vehicle to travel from $u\in V\cup \{0\}$ to $v\in V\cup \{0\}$, $u\ne v$ and $c_{uv}$ denotes the corresponding travel costs. 
Each task $v \in V$ is associated with its duration $d_v \; (0 \leq d_v \leq \Tmax)$, demand $q_v (0 \leq q_v \leq Q)$, and time window $[\alpha_v, \beta_v]$ during which the task can start $(0 \leq \alpha_v  \leq \beta_v \leq \Tmax)$.

Set $D\subseteq V^2$ contains all pairs of tasks with temporal dependencies between them and the set of all tasks with temporal dependencies is denoted as $\TDVis = \{v \in V: \exists \{u,v\} \in D\}$. Four non-negative parameters $\dmin{uv}$, $\dmax{uv}$, $\dmin{vu}$, and $\dmax{vu}$ are associated with each pair $\{u,v\}\in D$. Thereby, $\dmin{uv}$ and $\dmax{uv}$ specify the minimum and maximum allowed difference between the start of task $u$ and task $v$ when $u$ starts no later than $v$. Similarly, $\dmin{vu}$ and $\dmax{vu}$ correspond to the minimum and maximum allowed differences between the same tasks when $v$ starts before or simultaneously with $u$. \cref{tab:temp_dep_overview} summarizes parameter values for common temporal dependency types.
Using four parameters per pair accommodates a wide variety of complex temporal dependencies beyond the fundamental ones in \cref{tab:temp_dep_overview}.
These are realized by combining parameter values, e.g., a non-overlapping relation combined with a maximum starting time difference of $\Delta^\mathrm{max}$ can be modeled via $\dmin{uv}=d_u$, $\dmax{uv}=\Delta_\mathrm{max}$, $\dmin{vu}=d_v$, and $\dmax{vu}=\Delta_\mathrm{max}$.

\begin{table}
\TABLE
{Temporal dependencies and corresponding parameter values.\label{tab:temp_dep_overview}}
{\begin{tabular}{cccccc}
\toprule
Temporal dependency & $\dmin{uv}$ & $\dmax{uv}$ & $\dmin{vu}$ & $\dmax{vu}$ \\
\midrule
Synchronization & 0 & 0 & 0 & 0 \\
Minimum difference $(\Delta_{\mathrm{min}})$ & $\Delta_{\mathrm{min}}$ & $\Tmax$ & $\Delta_{\mathrm{min}}$ & $\Tmax$ \\
Maximum difference $(\Delta_{\mathrm{max}})$ & 
0 & $\Delta_{\mathrm{max}}$ & 0 & $\Delta_{\mathrm{max}}$ \\
Minimum and maximum difference &  
$\Delta_{\mathrm{min}}$ & $\Delta_{\mathrm{max}}$ & $\Delta_{\mathrm{min}}$ & $\Delta_{\mathrm{max}}$ \\
Overlap & 0 & $d_u$ & 0 & $d_v$ \\
Non-overlap & $d_u$ & $\Tmax$ & $d_v$ & $\Tmax$ \\
Precedence ($u$ before $v$, i.e., $u \preceq v$) & 0 & $\Tmax$ & $\Tmax$ & $\Tmax$ \\ 
\bottomrule
\end{tabular}}{}
\end{table}

The objective of the \probabbr is to design routes and schedules for the available vehicles that minimize the total travel cost while ensuring that all tasks are performed and all time windows, vehicle capacities, and temporal dependencies are satisfied.

\medskip

\cref{fig:example:inst:route} illustrates a solution to a \probabbr instance in which all task durations and travel times are equal to one. Time windows are given next to each task. 
Although exact starting times are omitted, we discuss bounds different from the time windows imposed by temporal dependencies. 
In this instance, tasks 1 and 13 must be performed simultaneously, task 15 cannot start more than six time units later than task 4, and task 9 has to start at least 4 time units after task 16. Each vehicle has sufficient capacity to perform all tasks. 
The synchronization of tasks $1$ and $13$ implies that task $1$ cannot start earlier than time $t=4$ although the vehicle performing route $(0,5,1,6,11,0)$ could arrive at task $1$ already at $t=3$. This illustrates the dependencies between the routes of different vehicles. 
Route $(0,4,9,8,15,0)$ illustrates that temporal dependencies can also require postponing the start of a task to satisfy a maximum time difference to a subsequent task within the same route. Although the vehicle could arrive at task $4$ at time $t=1$, it cannot start earlier than $t=2$, since the maximum time difference to task $15$ is six and task $15$ cannot start before time $8$ due to the time window of task $8$.
As a result, task $9$, which follows task $4$ cannot start before $t=4$, and due to the minimum difference of four time units between the start of tasks $9$ and $16$, the earliest feasible starting time of task $16$ within route $(0,12,3,10,16,0)$ increases from $7$ to $8$. 
This example highlights the competing scheduling incentives between consecutive tasks $4$ and $9$ within route $(0,4,9,8,15,0)$. Delaying the start of task $4$ increases the scheduling flexibility of task $15$ (due to the maximum time difference), whereas starting task $4$ earlier improves the scheduling flexibility of $16$ (due to the minimum time difference). 

 \begin{figure}
     \FIGURE
 {\begin{tikzpicture}[scale=0.11, every node/.style={scale=0.9, font=\footnotesize}]
			\includeStyleExample
			\includeNodesExample
			\drawRouteOne
			\drawRouteTwo
			\drawRouteThree
			\drawRouteFour
    \end{tikzpicture}}
{Visualization of a \probabbr instance and corresponding feasible solution.  \label{fig:example:inst:route}}
{}
\end{figure}

\paragraph{Assumptions and notation}
We assume that travel times and travel costs satisfy the triangle inequality.
To simplify notation, we assign time window $[\alpha_0,\beta_0]=[0,\Tmax]$, duration $d_0=0$, and quantity $q_0=0$ to the depot $0$ which may thus be referred to as task where convenient. We also define $N = V \cup \{0\}$ as the set of tasks including the depot.

\subsection{Fragment-based formulation}\label{ss:prob:def:frag:form}
In this section, we introduce a MILP for the \probabbr that is based on representing vehicle routes by sequences of so-called fragments. Fragments are ordered sequences of tasks that can be performed by a single vehicle and may contain tasks with temporal dependencies only at their start and end. Fragments start with performing their first task and end when arriving at their last task, see \cref{def:frag}. As discussed below, in contrast to routes there are no competing scheduling incentives for tasks performed within a fragment, i.e., delaying the start of a task contained within a fragment is never advantageous. 
As a result, a key feature of the proposed fragment-based formulation is its balance between the strength of its linear relaxation and the complexity of the pricing subproblem arising in a column-generation-based solution algorithm. 
Other work exploiting a compromise between an arc-based and route-based formulation includes, e.g., \cite{Dollevoet2025} for a capacitated VRP, \cite{Alyasiry2019, Sippel2024, SIPPEL2025107123, FORBES2026} for (variants of) the pickup-and-delivery problem, \cite{Rist2021, RIST2022105649} for (variants of) the dial-a-ride problem, \cite{Paradiso2020, Yang2023} for the multi-trip VRPTW, and \cite{DallOllio2023} for airport ground handling.
Fragment-based route representations are also used for the electric autonomous dial-a-ride problem to compute excess user ride time \citep{SU20231091} and as input for a labeling algorithm \citep{SU2024103011}.

\begin{definition}[Fragment]\label{def:frag}
A fragment $F=(v_1,\dots,v_\lenF{F})$ is an ordered sequence of tasks that 
\begin{inparaenum}[(i)]
    \item starts with the performance of task $v_1\in \TDVis\cup \{0\}$;
    \item ends when arriving at the location of task $v_\lenF{F}\in \TDVis\cup \{0\}$;
    \item does not contain any further task with temporal dependencies (i.e., $\{v_2, \dots, v_{\lenF{F}-1}\}\subseteq V\setminus \TDVis$); 
    \item contains each task different from the depot at most once; and
    \item can be performed by a single vehicle.
\end{inparaenum}
The last condition states that the total demand $\sum_{i=1}^{\lenF{F}} d_{v_{i}}$ of all tasks of fragment $F$ may not exceed the vehicle capacity $Q$ and that fragment $F$ admits a feasible schedule, i.e., there exist starting times $b_{v_i}\in [\alpha_{v_i},\beta_{v_i}]$ for each $i\in \{1, \dots, \lenF{F}\}$ such that $b_{v_i} + d_{v_i} + t_{v_i v_{i+1}} \le b_{v_{i+1}}$ holds for each $i=1,\dots, \lenF{F}-1$ and $b_{v_{\ell}} \in [b_{v_{1}} + \dmin{v_{1} v_{\lenF{F}}}, b_{v_{1}} + \dmax{v_{1} v_{\lenF{F}}}]$ if $\{v_{1}, v_{\lenF{F}}\} \in D$.
\end{definition}

\cref{def:frag} implies that fragments can be classified into the following four types according to their start and end tasks: 
\begin{inparaenum}[(i)]
    \item from the depot to a task with temporal dependencies; 
    \item from a task with temporal dependencies to another task with temporal dependencies;
    \item from a task with temporal dependencies to the depot;
    \item from the depot to the depot.
\end{inparaenum}
Each feasible vehicle route and, therefore, also each solution to the \probabbr has a unique representation as a set of fragments.
\cref{fig:example:inst:frag} visualizes the representation of the solution given in \cref{fig:example:inst:route} by its fragments. The combination of fragments $(0,5,\textbf{1})$ and $(1,6,11,0)$ forms route $(0,5,1,6,11,0)$, while the fragment sequence $(0,2, \textbf{13})$ and $(13,7,14,0)$ represents route $(0,2,13,7,14,0)$. Route $(0,12,3,10,16,0)$ corresponds to fragments $(0,12,3,10,\textbf{16})$ and $(16, 0)$, while the fragments $(0,\textbf{4})$, $(4,\textbf{9})$, $(9,8,\textbf{15})$, and $(15,0)$ form route $(0,4,9,8,15,0)$.

\begin{figure}
	\FIGURE
	{\begin{tikzpicture}[scale=0.11, every node/.style={scale=0.9, font=\footnotesize}]
			\includeStyleExample
			\includeSynNodesExample
			\drawRouteOneFragForm
			\drawRouteTwoFragForm
			\drawRouteThreeFragForm
			\drawRouteFourFragForm
	\end{tikzpicture}}
	{Visualization of the solution of \cref{fig:example:inst:route} by fragments.    \label{fig:example:inst:frag}}
	{}
\end{figure}

Each fragment $F=(v_1, \dots, v_\lenF{F})$ is associated with its total travel cost $c_F=\sum_{i=1}^{\lenF{F}-1} c_{v_i v_{i+1}}$ and its total 
demand $q_F=\sum_{i=1}^{\lenF{F}-1} q_i$ of all tasks served by $F$. 
Fragments are linked to feasible starting times of their tasks through the concept of (compact) schedules, formally introduced in \cref{def:schedule}, which is crucial to understand the advantages of our fragment-based approach and to introduce further concepts such as the latest starting time $\lsF{v_1}{F}$, earliest completion time $\esF{v_\ell}{F}$, and minimum duration $\durF{F}$ of a fragment $F$. 

\begin{definition}[(Compact) Schedule]\label{def:schedule}
	A \emph{schedule} $(b_1, \dots, b_\lenF{F})$ of a fragment $F=(v_1, \dots, v_\lenF{F})$ is an ordered set of starting times for all tasks $v_1, \dots, v_{\lenF{F}-1}$ performed in fragment $F$ together with a lower bound $b_\lenF{F}$ for the starting time of task $v_{\lenF{F}}$ according to which fragment $F$ can be performed by a vehicle, i.e., such that $b_i\in [\alpha_{v_i},\beta_{v_i}]$ for $i\in \{1, \dots, \lenF{F}\}$, $b_i + d_{v_{i}} + t_{v_{i} v_{i+1}} \le b_{i+1}$ for $i\in \{1, \dots, \lenF{F}-1\}$, and $b_\lenF{F}\in [b_1+\dmin{v_1 v_\lenF{F}}, b_1+\dmax{v_1 v_\lenF{F}}] \mbox{ if } \{v_1, v_\lenF{F}\}\in \TDVis$. Each schedule $(b_1, \dots, b_\lenF{F})$ has a \emph{duration} equal to $b_{\lenF{F}}- b_{1}$. 
	A schedule $(b_1, \dots, b_\lenF{F})$ 
    for fragment $F$ is a \emph{compact schedule} when its duration is minimal among all schedules for $F$ starting at $b_1$ and in which all intermediate tasks $v_i$, $i=2, \dots, \lenF{F}$, are started as early as possible, i.e., $b_i=\max \{ b_{i-1} + d_{v_{i-1}} + t_{v_{i-1} v_i}, \alpha_{v_i}\}$ holds for all $i=2, \dots, \lenF{F}-1$, and 
	$b_{v_\lenF{F}}=\max \{ b_{\lenF{F}-1} + d_{v_{\lenF{F}-1}} + t_{v_{\lenF{F}-1} v_\lenF{F}}, \alpha_{v_\lenF{F}},b_1+\dmin{v_1,v_\lenF{F}}\}$.
    The \emph{set of all compact schedules} of 
    $F$ is denoted as $\schedF{F}$.
\end{definition}

One key advantage of using fragments instead of full routes is that they allow decoupling the determination of the exact starting time of tasks with temporal dependencies and the identification of feasible sequences of tasks without temporal dependencies. 
Since each fragment $F=(v_1, \dots, v_{\lenF{F}})$ performs at most one task with temporal dependency, there is no incentive to introduce additional waiting time between or before the intermediate tasks $v_2, \dots, v_{\lenF{F}-1}$ that do not have temporal dependencies. Indeed, it can never be beneficial to delay intermediate tasks within a fragment as this can only increase the earliest completion time of the fragment which is a lower bound for the starting time of $v_{\lenF{F}}$. As stated in \cref{obs:compact-schedules}, it therefore suffices to consider compact schedules with maximum scheduling flexibility for tasks with temporal dependencies. The resulting absence of conflicting scheduling incentives within fragments drastically simplifies the comparison of (partial) fragments and is, therefore, crucial for the efficiency of our exact price-cut-and-enumerate solution method, see \cref{ss:solappr:init:lb} and \cref{ss:solappr:frag:enum}.

\begin{observation}\label{obs:compact-schedules}
	For each instance of the \probabbr there exists an optimal solution in which each fragment is performed according to a compact schedule.
\end{observation}

There is exactly one compact schedule for each feasible starting time of a fragment, i.e., of its first task. Thus, the starting times of all intermediate tasks of a fragment are implied by its starting time when fragments are performed according to a compact schedule (which we can assume w.l.o.g.).
The latest feasible 
starting time of fragment $F=(v_1,\dots,v_{\lenF{F}})$ is equal to $\lsF{v_1}{F}=\max \{b_1\mid (b_1, \dots, b_\lenF{F})\in \schedF{F}\}$, its earliest completion time is given by $\esF{v_\lenF{F}}{F}=\min \{b_\lenF{F} \mid (b_1, \dots, b_\lenF{F})\in \schedF{F}\}$, and its minimum duration can be stated as $\durF{F}=\min\{b_\lenF{F}-b_1 \mid (b_1, \dots, b_\lenF{F})\in \schedF{F}\}$. 
As shown in \cref{prop:frag:derive:sched:cond} these values can be derived in linear time. This follows from a generalization of the results of \cite{Desaulniers2000} and the labeling algorithms specified in \cite{Irnich2008, Tilk2017, Paradiso2020} to our setting where temporal dependencies can only occur between the start and end task of a path. 

\begin{restatable}{proposition}{fragderiveschedcond}\label{prop:frag:derive:sched:cond}
	The latest starting time $\lsF{v_1}{F}$, earliest completion time $\esF{v_{\lenF{F}}}{F}$, and duration $\durF{F}$ of a fragment $F=(v_1, \dots, v_{\lenF{F}})$ can be determined in $\mathcal{O}(\lenF{F})$ by solving the dynamic programming recursions
	\begin{subequations}\label{eq:prop:frag:derive:sched:cond}
		\begin{align}
		\esF{v_i}{F[i]} & = \max\{\esF{v_{i-1}}{F[i-1]} + d_{v_{i-1}} + t_{v_{i-1} v_{i}}, \alpha_{v_{i}} \} \\
        \lsF{v_{1}}{F[i]} & =  \min\{ \lsF{v_{1}}{F[i-1]}, \beta_{v_i} - t_{v_{i-1} v_{i}} - d_{v_{i-1}} - \durF{F[i-1]}\} \\
		\durF{F[i]} & = \max\{ \durF{F[i-1]} + d_{v_{i-1}} + t_{v_{i-1} v_{i}}, \alpha_{v_{i}} - \lsF{v_1}{F[i-1]}\}		
		\end{align}
	\end{subequations}
	for $i=2,\dots,\lenF{F}$. If $\{v_1,v_\lenF{F}\} \in D$ and $\durF{F[\lenF{F}]} < \dmin{v_1 v_\lenF{F}}$ the additional updates 
    $
    \esF{v_{\lenF{F}}}{F[\lenF{F}]} = \max \{\esF{v_{\lenF{F}}}{F[\lenF{F}]}, \alpha_{v_1} + \dmin{v_1 v_{\lenF{F}}} \}$, $\lsF{v_{1}}{F[\lenF{F}]} = \min \{\lsF{v_{1}}{F[\lenF{F}]}, \beta_{v_{\lenF{F}}} - \dmin{v_1 v_{\lenF{F}}}\}$,  and $\durF{F[\lenF{F}]} = \max \{\durF{F[\lenF{F}]}, \dmin{v_1 v_{\lenF{F}}} \}$
    are required.
	Here, $F[i]=(v_1, \dots, v_i)$ is the partial fragment of $F$ until task $i\in \{1, \dots, \lenF{F}\}$. The initial values are set to 
	$\esF{v_{1}}{F[1]}=\alpha_{v_1}$, $\lsF{v_1}{F[1]}=\beta_{v_1}$, and $\durF{F[1]}=0$ and final quantities are obtained as $\lsF{v_{1}}{F}= \lsF{v_{1}}{F[\lenF{F}]}, \esF{v_{\lenF{F}}}{F} = \esF{v_{\lenF{F}}}{F[\lenF{F}]}$, and $\durF{F} = \durF{F[\lenF{F}]}$. 
\end{restatable}

From \cref{prop:frag:derive:sched:cond} and its proof, we observe that for a fragment $F$ a schedule with duration $\durF{F}$ can be achieved for each starting time in $[\esF{v_{\lenF{F}}}{F} - \durF{F}, \lsF{v_1}{F}]$. \cref{cor:frag:duration} follows from this observation and the fact that the additional mandatory waiting time resulting from starting a fragment at a time earlier than the one resulting in a minimum duration is equal to the difference between this time and the actual starting time.

\begin{restatable}{corollary}{fragdur}\label{cor:frag:duration}
    The duration of a compact schedule for fragment $F=(v_1, \dots, v_{\lenF{F}})$ that starts at feasible starting time $t\in [\max\{\alpha_{v_1},\esF{v_\lenF{F}}{F} - \dmax{v_1 v_{\lenF{F}}}\}, \lsF{v_\lenF{F}}{F}]$ is equal to 
	\[
	 \durF{F}(t) = \begin{cases}
	 \durF{F} & \mbox{if $t\in [\esF{v_{\lenF{F}}}{F} - \durF{F}, \lsF{v_1}{F}]$ }\\
		\esF{v_{\lenF{F}}}{F} - t & \mbox{ otherwise } \\					
	\end{cases}.
	\]
\end{restatable}

The minimum duration of a fragment is crucial for enforcing timing and temporal dependency constraints in our fragment-based formulation~\eqref{eq:ff} which considers the set of all fragments $\AllFrag$ and associates a 
variable $x_F\in \{0,1\}$ to each fragment $F\in \AllFrag$.
These variables indicate whether fragment $F\in \AllFrag$ is selected. 
We use $\FragTo{v}$ and $\FragFrom{v}$ to denote all fragments starting and ending at task $v \in \TDVis$, respectively. Similarly, $\FragFromTo{S}{S'}$ denotes the set of fragments starting at $S\subseteq \{0\} \cup \TDVis$ and ending at $S'\subseteq \{0\} \cup \TDVis$. We use $\textrm{end(F)}$ to denote the last task of a fragment $F \in \AllFrag$. 

\begin{subequations}\label{eq:ff}
\begin{align}
\min \quad   & \sum_{F \in \AllFrag} \costF{F} x_{F} \label{eq:ff:obj}  \\
\mbox{s.t.}\quad & \sum_{F \in \FragFrom{0}} 
     x_{F} \leq |\Veh| 
     \label{eq:ff:vehicle:limit} \\ 
& \sum_{\substack{F \in \AllFrag : v \in F \\ v \neq \textrm{end(F)}}} x_{F}  = 1 & \forall v \in V   
\label{eq:ff:ass} \\
& \sum_{F \in \FragTo{v}} x_{F}  - \sum_{F \in \FragFrom{v}} x_{F} = 0 & \forall v\in \TDVis   
\label{eq:ff:flow} \\ 
& \mbox{timing constraints } \label{eq:ff:tc}\\
& \mbox{temporal-dependency constraints} \label{eq:ff:tdc}\\
& \mbox{capacity constraints} \label{eq:ff:cc}\\
& x_{F} \in  \{0,1\} & \forall F \in \AllFrag \label{eq:ff:x}
\end{align}
\end{subequations}

The objective \eqref{eq:ff:obj} minimizes the total travel cost. Constraints \eqref{eq:ff:vehicle:limit} ensure that at most $|\Veh|$ 
vehicles are used and equations \eqref{eq:ff:ass} make sure that each task is performed. Flow conservation constraints~\eqref{eq:ff:flow} connect consecutive fragments at tasks with temporal dependencies and ensure that the selected fragments form routes.  
While each individual fragment corresponds to a feasible partial route, we need to ensure that their combinations correspond to routes starting and ending at the depot that allow for feasible schedules with respect to task starting times (\emph{timing constraints} \eqref{eq:ff:tc}), temporal dependencies (\emph{temporal-dependency constraints} \eqref{eq:ff:tdc}), and vehicle capacities (\emph{capacity constraints} \eqref{eq:ff:cc}). 
In the following, we discuss how to enforce these conditions by means of additional variables and constraints replacing \eqref{eq:ff:tc}--\eqref{eq:ff:cc}.
Alternatively, infeasible fragment combinations could be eliminated through cutting planes by adapting so-called infeasible path inequalities which are, however, known to lead to weak formulations, see, e.g., \cite{Rist2021}.

\paragraph{Timing constraints}
A sequence of fragments can form a feasible route if each fragment can start after all its predecessors have been completed. Since a feasible starting time of a fragment implies that there is a schedule such that all its tasks can be performed within their time windows, it is sufficient to impose timing restrictions at tasks with temporal dependencies. Therefore, variables $b_v\ge 0$ representing the starting times of tasks with temporal dependencies $v\in \TDVis$ are used in timing constraints \eqref{eq:ff:time:mtz}-\eqref{eq:ff:time:ub} used in place of abstract condition~\eqref{eq:ff:tc}.
\begin{subequations}\label{eq:ff:time:all}
\begin{align}
&  b_{u} + \sum_{F \in \FragFromTo{u}{v}} \durF{F} x_F  \leq   b_{v} + (\beta_{u}-\alpha_v) (1- \sum_{F \in \FragFromTo{u}{v}} x_F) & \forall u,v \in \TDVis 
\label{eq:ff:time:mtz} \\  
&b_v - \sum_{F \in \FragTo{v}} \esF{v}{F} x_F  \ge 0 & \forall v\in \TDVis  \label{eq:ff:time:lb} \\   
&- b_v + \sum_{F \in \FragFrom{v}} \lsF{v}{F} x_F  \ge 0 & \forall v\in \TDVis \label{eq:ff:time:ub} 
\\ 
& b_{v}  \geq  0 & \forall v \in \TDVis \label{eq:ff:time:b} 
\end{align}
\end{subequations}
Inequalities~\eqref{eq:ff:time:mtz} ensure that the starting time difference between two tasks with temporal dependencies is at least the duration of the fragment connecting them. 
Constraints \eqref{eq:ff:time:lb} and \eqref{eq:ff:time:ub} enforce bounds on the starting times of tasks with temporal dependencies based on the selected ingoing and outgoing fragments. Together, they ensure that the duration of each fragment matches the one given by \cref{cor:frag:duration}.

\paragraph{Temporal-dependency constraints \eqref{eq:ff:tdc}}
Conditions~\eqref{eq:ff:tdc} are enforced by constraints~\eqref{eq:ff:td:uv:le}-\eqref{eq:ff:td:p} which ensure that the temporal restrictions between each pair of tasks $\{u,v\}\in D$ are satisfied.
Here, variables $p_{uv} \in \{0,1\}$ indicate the order of tasks $\{u,v\}\in D$, where $p_{uv} = 1$ indicates that $u$ starts before or at the same time as $v$.
\begin{subequations}
\begin{align}       
& b_v - b_u \le \delta^{\mathrm{max}}_{uv} p_{uv}           & \forall \{u,v\}\in D, u<v \label{eq:ff:td:uv:le}\\  
& b_u - b_v \le \delta^{\mathrm{max}}_{vu}  (1-p_{uv})              & \forall \{u,v\}\in D, u<v \label{eq:ff:td:vu:le} \\             
& b_v - b_u \ge \delta^{\mathrm{min}}_{uv} p_{uv} - \max\{0,\beta_u-\alpha_v\} (1 - p_{uv})   &          \forall \{u,v\}\in D, u<v \label{eq:ff:td:uv:ge}\\     
& b_u - b_v \ge \delta^{\mathrm{min}}_{vu} (1 - p_{uv})  -\max\{0,\beta_v-\alpha_u\} p_{uv} & \forall \{u,v\}\in D, u<v \label{eq:ff:td:vu:ge}  \\ 
& p_{uv}  \in  \{0,1\} & \forall  \{u,v\}\in D, u<v  \label{eq:ff:td:p}
\end{align}
\end{subequations} 
Inequalities \eqref{eq:ff:td:uv:le} and \eqref{eq:ff:td:vu:le} enforce the maximum allowed starting time difference ($\dmax{uv}$ or $\dmax{vu}$) for a pair of tasks $\{u,v\} \in D$. 
Constraints \eqref{eq:ff:td:uv:ge} and \eqref{eq:ff:td:vu:ge} ensure the minimum difference ($\dmin{uv}$ or $\dmin{vu}$) between their starting times. 
The order variables $p_{uv}$ activate the appropriate inequality depending on the order in which tasks  $\{u,v\}\in D$ are performed.

\paragraph{Capacity constraints \eqref{eq:ff:cc}}
Constraints~\eqref{eq:ff:load:mtz}-\eqref{eq:ff:load:l} ensure that no route exceeds the vehicle capacity. As for timing constraints~\eqref{eq:ff:time:all}, it is sufficient to impose load restrictions at tasks with temporal dependencies. 
This is achieved using variables $l_{u} \geq 0$ representing the total demand served in the corresponding route until task $u \in \TDVis$.
\begin{subequations}\label{eq:ff:load:all}
\begin{align}       
& l_{u} + \sum_{F \in \FragFromTo{u}{v}} q_{F} x_F \leq   l_{v} +(\VehCap - q_u) (1- \sum_{F \in \FragFromTo{u}{v}} x_F)  & \forall u,v \in \TDVis 
\label{eq:ff:load:mtz}  \\ 
&l_{v} - \sum_{F \in \FragTo{v}} q_{F} x_F   \geq 0    & \forall v \in \TDVis  
\label{eq:ff:load:lb}   \\
&  \VehCap -l_{v} - \sum_{F \in \FragFrom{v}} q_{F} x_F  \geq 0 & \forall v \in \TDVis \label{eq:ff:load:ub} \\
& l_{v}  \geq  0 & \forall v \in \TDVis \label{eq:ff:load:l}
\end{align}
\end{subequations}
Constraints \eqref{eq:ff:load:mtz} ensure that the load of a fragment is accounted for when tasks $u,v\in \TDVis$ are connected via a fragment. 
Inequalities \eqref{eq:ff:load:lb} and \eqref{eq:ff:load:ub} bound the served demand at tasks with temporal dependencies.

\section{Valid inequalities}\label{sec:valineq}
In this section, we propose several families of valid inequalities (VI) that exploit the definition of fragments. While the constraints introduced in \cref{sec:fsec} share similarities with classical subtour elimination constraints, those in \cref{sec:tifi,sec:tdifi} generalize path-elimination constraints previously used for routing problems with synchronization constraints 
\citep{Luo2016}. Finally, the inequalities introduced \cref{sec:fcc} are based on (rounded) capacity constraints.

\subsection{Fragment-based subtour elimination constraints}\label{sec:fsec}
Subtour elimination constraints are frequently used to ensure connectivity in vehicle routing formulations.
The fragment-based subtour elimination constraints (FSECs) introduced below eliminate subtours implied by sets of fragments while accounting for the minimum number of vehicles required to serve a given set of temporal dependency tasks. Let $S\subseteq \TDVis$ be a set of tasks with temporal dependencies and $\FSECVehMin{S}$ be the minimum number of vehicles required to perform all tasks in $S$. This quantity depends on the time windows of tasks, temporal dependencies between them, their demands, and the vehicle capacity.
FSECs~\eqref{eq:FSEC} ensure that the subgraph induced by all fragments connecting tasks from $S$ consists of at least $\FSECVehMin{S}$ components. 
\begin{align}\label{eq:FSEC}
        & \sum_{F \in \FragFromTo{S}{S}} x_F \leq |S| -\FSECVehMin{S} &  \forall S \subseteq \TDVis
\end{align}

FSECs enforce that every feasible solution contains at least $\FSECVehMin{S}$ fragments entering set $S \subseteq \TDVis$, since each task with a temporal dependency must have exactly one ingoing and one outgoing fragment (see \eqref{eq:ff:ass} and \eqref{eq:ff:flow}). 
We note that computing $\FSECVehMin{S}$ for a given set $S \subseteq \TDVis$ and, therefore, separating FSECs is NP-hard. This follows since it is equivalent to the NP-hard bin-packing problem \citep{Garey1979} when time windows and temporal dependencies are relaxed.

\subsection{Time-infeasible fragment inequalities}\label{sec:tifi}
The inequalities proposed in this section eliminate combinations of fragments incident to a task with temporal dependencies whose simultaneous usage would violate the timing constraints. Consider a task $v\in \TDVis$ and two fragments $F\in \FragTo{v}$ and $F'\in \FragFrom{v}$ that end and start at task $v$, respectively. Fragments $F$ and $F'$ are incompatible if the earliest completion time of $F$ at task $v$ is larger than the latest departure time of $F'$. Generalizing this observation to two sets of pairwise incompatible fragments associated with a task $v\in \TDVis$ yields the Time-Infeasible Fragment Inequalities (TIFI) 
\begin{equation}\label{eq:InfPathV}
    \sum_{\substack{F \in \FragTo{v} :\\ \esF{v}{F} \geq t}} x_{F} + \sum_{\substack{F \in \FragFrom{v} :\\ \lsF{v}{F} < t}} x_F\leq 1 \quad \forall v \in \TDVis,\ t \in [\alpha_v, \beta_v].
\end{equation}

An upper bound on the number of relevant TIFIs is given in \cref{thm:InfPathV:num} which characterizes the set of time points for which non-dominated TIFIs may be identified. 

\begin{restatable}{theorem}{numtifi}\label{thm:InfPathV:num}
For each task $v\in \TDVis$, it is sufficient to consider TIFIs at time points $t\in \{\esF{v}{F} : F \in \FragTo{v}\}$ to obtain all non-dominated TIFIs.
\end{restatable}

\subsection{Temporal dependency infeasible fragment inequalities}\label{sec:tdifi}
In this section, we propose inequalities that eliminate combinations of fragments incident to pairs of tasks $\{u,v\}\in D$ whose simultaneous usage would violate the temporal dependencies between the two tasks. Analogous to the case of single tasks discussed above, an ingoing fragment $F\in \FragTo{u}$ of task $u$ may be incompatible with an outgoing fragment $F'\in \FragFrom{v}$ of task $v$ if their latest starting and earliest completion times do not allow satisfying the temporal dependency between the two tasks. Whether two fragments are incompatible may, however, also depend on the order of the two tasks (i.e., whether task $u$ precedes task $v$) which is typically not known in advance. Indeed, a pair of fragments $F\in \FragTo{u}$ and $F'\in \FragFrom{v}$ 
where $u$ cannot start before $\esF{u}{F}$ and $v$ cannot start after $\lsF{v}{F'} < \esF{v}{F} + \dmin{uv}$
is incompatible only if task $u$ precedes task $v$. This is captured by the inequality
\begin{equation}
x_F + x_{F'} \le 2 - p_{uv}.\label{eq:TDIFI}
\end{equation} 

The set of temporal dependency infeasible fragment inequalities~\eqref{eq:InfPathTD} referred to as TDIFIs is obtained by: \begin{inparaenum}[(i)]
	\item lifting from a single pair to sets of pairwise incompatible fragments; 
	\item considering maximum in addition to minimum time differences; and
	\item covering the case when $v$ precedes $u$. 
\end{inparaenum}
\begin{subequations}\label{eq:InfPathTD}
\begin{align}
    & \sum_{\substack{F \in \FragTo{u} : \\ \esF{u}{F} \geq t }} x_{F} + \sum_{\substack{F \in \FragFrom{v} :\\ \lsF{v}{F} < t+\dmin{uv}}} x_{F} \leq 2 - p_{uv}  & \forall \{u,v\} \in D, u<v, t \in [\alpha_u, \beta_u]  \label{eq:InfPathTD:uv:min}\\
    & \sum_{\substack{F \in \FragFrom{u} :\\ \lsF{u}{F} \leq t}} x_{F} + \sum_{\substack{F \in \FragTo{v}:\\ \esF{v}{F} > t+\dmax{uv}}} x_{F} \leq 2 - p_{uv}  & \forall \{u,v\} \in D, u<v, t \in [\alpha_u, \beta_u] \label{eq:InfPathTD:uv:max}\\
    & \sum_{\substack{F \in \FragTo{v} :\\ \esF{v}{F} \geq t}} x_{F} + \sum_{\substack{F \in \FragFrom{u}:\\ \lsF{u}{F} < t+\dmin{vu}}} x_{F} \leq 1 + p_{uv}  &\forall \{u,v\} \in D, u<v, t \in [\alpha_v, \beta_v] \label{eq:InfPathTD:vu:min}\\ 
    & \sum_{\substack{F \in \FragFrom{v} :\\ \lsF{v}{F} \leq t}} x_{F} + \sum_{\substack{F \in \FragTo{u} : \\  \esF{u}{F} > t+\dmax{vu}}} x_{F} \leq 1 + p_{uv} & \forall \{u,v\} \in D, u<v, t \in [\alpha_v, \beta_v] \label{eq:InfPathTD:vu:max}
\end{align}
\end{subequations}

Inequalities~\eqref{eq:InfPathTD:uv:min} exclude fragment combinations that are incompatible with respect to  the minimum starting-time difference $\dmin{uv}$ between $u$ and $v$ if $u$ precedes $v$. They generalize constraint~\eqref{eq:TDIFI} by considering each relevant point in time $t\in [\alpha_u,\beta_u]$ and lifting it using sets of pairwise incompatible fragments, i.e., adding all ingoing fragments of $u$ that cannot finish before $t$ and all outgoing fragment of $v$ that must start before $t+ \dmin{uv}$. Inequalities \eqref{eq:InfPathTD:uv:max} focus on the maximum allowed difference $\dmax{uv}$ between the starting time of $u$ and $v$. They enforce that each solution may include at most one fragment from the sets $\{F \in \FragFrom{u}: \lsF{u}{F} \leq t \}$ and $\{F \in \FragTo{v} : \esF{v}{F} > t + \dmax{uv}\}$ when $p_{uv} = 1$. These inequalities are valid since using a fragment from the former set implies that task $u$ cannot start later than $t$ while using a fragment from the latter set implies that task $v$ has to start later than $t+\dmax{uv}$.
Constraints \eqref{eq:InfPathTD:vu:min} and \eqref{eq:InfPathTD:vu:max} are analogous to \eqref{eq:InfPathTD:uv:min} and \eqref{eq:InfPathTD:uv:max} for the case when $v$ start before or at the same time as $u$, i.e., they forbid certain fragment combinations when $p_{uv}=0$.

\cref{thm:InfPathTD:num} establishes an upper bound on the number of relevant TDIFIs and follows from arguments analogous to the ones in the proof of \cref{thm:InfPathV:num}.

\begin{restatable}{theorem}{numtdifi}\label{thm:InfPathTD:num}
For each pair of tasks $\{u,v\}\in D$ with temporal dependencies, it is sufficient to consider time points from $\{\esF{u}{F} : F \in \FragTo{u}\} $, $\{\lsF{u}{F} : F \in \FragFrom{u}\}$, $\{\esF{v}{F} : F \in \FragTo{v}\}$, and $t\in \{\lsF{v}{F} : F \in \FragFrom{v}\}$ to obtain all non-dominated inequalities \eqref{eq:InfPathTD:uv:min}, \eqref{eq:InfPathTD:uv:max}, \eqref{eq:InfPathTD:vu:min}, and \eqref{eq:InfPathTD:vu:max}, respectively.
\end{restatable}

A stronger version of inequalities~\eqref{eq:InfPathTD:uv:max} can be obtained by reducing their right-hand side from $2-p_{uv}$ to $1$. For a given pair $\{u,v\} \in D$ the left-hand side value of inequalities~\eqref{eq:InfPathTD:uv:max} cannot exceed $1$ in case $p_{uv}=0$ since it contains only pairs of fragments from $\FragFrom{u}$ and $\FragTo{v}$ that must start at task $u$ earlier than arriving at task $v$. 
The right-hand side of inequalities~\eqref{eq:InfPathTD:vu:max} can be set to $1$ following a similar argument.

\subsection{Fragment-based (rounded) capacity constraints}\label{sec:fcc}

(Rounded) capacity constraints are commonly used in edge- or arc-based models for capacitated routing problems \citep{Laporte1983, Lysgaard2004}. Given a set of tasks $S\subseteq V$, they ensure that the number of vehicles visiting at least one task in $S$ is sufficient to serve the total demand of all tasks in $S$. In the space of fragment variables, this can be enforced by fragment-based rounded capacity constraints (FRCCs) 
\begin{align}
	& \sum_{F\in \delta^-_\AllFrag(S)} x_F \ge \left \lceil \frac{\sum_{v\in S} q_v}{Q} \right \rceil & \forall S\subseteq V, \label{eq:FRCC}
\end{align}
where $\delta^{-}_\AllFrag(S)=\{(v_1, \dots, v_\ell)\in \AllFrag : v_1\notin S, (\cup_{i=1}^\ell \{v_i\}) \cap S \ne \emptyset\}$ denotes the set of fragments that start outside of $S$ and visit at least one task from $S$. Their validity follows from the observation that the number of fragments entering $S$ at least once is an upper bound on the number of vehicles that perform tasks from $S$.

Obviously, FRCCs are closely related to (standard) rounded capacity constraints
\begin{equation}\label{eq:RCC}
    \sum_{u \in S} \sum_{v \in N \setminus S} \sum_{\substack{F \in \AllFrag : \\ v = succ(u)}}
     x_{F}  \geq  \left \lceil \frac{\sum_{v \in S} q_v}{\VehCap} \right \rceil \quad \forall S \subseteq V,
\end{equation}
in which $succ(u)$ denotes the successor of task $u$ in fragment $F$.
Notice that the left-hand side of FRCCs~\eqref{eq:FRCC} contains every fragment entering set $S$ only once, while the coefficient of each such fragment on the left-hand side of RCCs~\eqref{eq:RCC} is equal to the number of times it enters set $S$. Thus, each fragment-based rounded capacity constraint~\eqref{eq:FRCC} dominates the rounded capacity constraint~\eqref{eq:RCC} defined for the same set of nodes.

\section{Solution approach}\label{sec:solappr}

This section describes our exact fragment-based solution algorithm (\solalg) for the \probabbr based on price-cut-and-enumerate. \solalg first uses variants of the proposed fragment-based formulation enhanced with valid inequalities to compute an initial lower bound ($\LBalg$) and upper bound ($\IncSolVal$) on the optimal solution. Subsequently, it iteratively refines these bounds through fragment enumeration, fragment reduction, and branch-and-cut until it either converges to a proven optimal solution or reaches a given time limit. The algorithm maintains the set of fragments contained in the best solution found so far ($\IncSol$) and a candidate upper bound ($\UBguess$), which depends on parameters $\UBinitstep$ and $\UBstep$. Throughout this description of \solalg, \FF denotes formulation~\eqref{eq:ff} with timing constraints \eqref{eq:ff:time:mtz}-\eqref{eq:ff:time:b}, temporal dependency constraints \eqref{eq:ff:td:uv:le}-\eqref{eq:ff:td:p}, and capacity constraints \eqref{eq:ff:load:mtz}-\eqref{eq:ff:load:l}. Notation \FFVI refers to \FF strengthened by inequalities \eqref{eq:FSEC}, \eqref{eq:InfPathV}, \eqref{eq:InfPathTD}, and \eqref{eq:RCC}, whereas \FFVIm denotes \FF strengthened by constraints \eqref{eq:FSEC}, \eqref{eq:InfPathV}, \eqref{eq:InfPathTD}, and \eqref{eq:FRCC}.
Furthermore, for each considered formulation $\mathcal{X}$ we use $\REL{\mathcal{X}}$ to refer to its linear relaxation, $\FORM{\mathcal{X}}[\AllFrag']$ to denote the variant of $\mathcal{X}$ restricted to the set of fragments $\AllFrag'\subseteq \AllFrag$, and $\VAL[\mathcal{X}]$ to denote the solution value obtained by solving it (which is assumed to be $+\infty$ if no feasible solution is found).

The main steps of \solalg are as follows:

\begin{enumerate}
    \item Preprocessing: Add implied temporal dependencies and update time windows; see \cref{app:pre:procc}.
    \item Lower Bound: Solve \relFFVI 
    via column-and-row generation to obtain the initial lower bound $\LBalg$ and the associated set of fragments $\CGfrag$; see \cref{ss:solappr:init:lb}. \label{sol:proc:step:lb}
    \item Upper Bound: Solve $\FFVIm[\CGfrag]$ via branch-and-cut 
    to obtain an upper bound $\IncSolVal=\VAL[\FFVIm[\CGfrag]]$ and fragments $\IncSol$; see \cref{ss:solappr:MIP}. Set $\UBguess= \min \{\VAL[\FFVIm[\CGfrag]], (1+\UBinitstep) \LBalg\}$. \label{sol:proc:step:init:ub}
    \item Fragment Enumeration: Enumerate all fragments $\ENfrag$ with reduced cost not exceeding $\UBguess - \VAL[\relFFVI[\CGfrag]]$ with respect to the dual values obtained in Step \ref{sol:proc:step:lb}, and remove fragments that cannot contribute to a solution with value below $\UBguess$; see \cref{ss:solappr:frag:enum}. \label{sol:proc:step:enum}
    \item Branch-and-Cut: Solve $\FFVIm[\ENfrag \cup \IncSol]$ with branch-and-cut; see \cref{ss:solappr:MIP}. If $\VAL[\FFVIm[\ENfrag \cup \IncSol]] < \IncSolVal$, set $\IncSolVal = \VAL[\FFVIm[\ENfrag \cup \IncSol]]$ and update $\IncSol$. \label{sol:proc:step:MIP}
    \item Termination Check: Set $\LBalg = \min\{\UBguess, \IncSolVal\}$. If $\IncSolVal \leq \LBalg$, the current solution is optimal and the algorithm terminates.  Otherwise, update the candidate upper bound to $\UBguess= \min\{\UBguess + \UBstep \cdot \VAL[\relFFVI[\CGfrag]], \IncSolVal\}$ and return to Step~\ref{sol:proc:step:enum}.
    \label{sol:proc:recursion}
\end{enumerate}
\medskip
\solalg follows the principle of enumerating all columns possibly required in an optimal solution as introduced by \cite{Baldacci2008,Baldacci2011}.
We note that our fragment-based formulation could also be solved using a branch-price-and-cut algorithm and that the computational attractiveness of our enumeration-based method mainly depends on the number of fragments that must be enumerated. 
This number, and thus the computation times and memory requirements, highly depends on the gap between the lower and upper bounds known prior to the enumeration step and the fragment length. In that respect, the definition of fragments naturally balances the fragment length and the formulation strength, limiting the number of fragments that must be enumerated. When there are few temporal dependencies, fragments represent almost complete routes, yielding a strong formulation and thus limiting the number of fragments that must be enumerated. Conversely, when most tasks have temporal dependencies, fragments become very short, typically corresponding to a single task and consecutive travel, in which case the number of fragments is limited due to their short length. As the results of our computational study will show, our algorithm obtains tight lower bounds from the column-and-row generation step, while high-quality solutions and tight upper bounds are typically found in Step~\ref{sol:proc:step:init:ub}. Nevertheless, a candidate upper bound, which can be iteratively increased if needed, is used as a fallback option for the few cases in which the enumeration step would require too much time or memory due to a larger gap between the initial lower and upper bounds. 

The pre-processing described in \cref{app:pre:procc} adds explicit temporal restrictions for task pairs that are linked through (a chain of) temporal dependency constraints, and the network of temporal dependencies is used to (potentially) strengthen each of them. Moreover, the pre-processing tightens the time windows of tasks by accounting for travel time from and to the depot as well as temporal dependency restrictions.

\subsection{Lower Bound Computation}\label{ss:solappr:init:lb}
This subsection describes the column-and-row generation procedure used to obtain the initial lower bound.  
\cref{sss:solappr:lb:mp} presents the master problem, \cref{sss:solappr:lb:subproblem} defines the pricing problem, \cref{sss:solappr:lb:dynamic:program} introduces an algorithm for solving the pricing problem, and \cref{sss:solappr:lb:acce} presents acceleration techniques. Finally, \cref{sss:solappr:lb:vi} discusses the separation of valid inequalities and their impact on the labeling algorithm.  

\subsubsection{Master Problem}\label{sss:solappr:lb:mp}

The column-and-row generation process for solving the linear relaxation of \FFVI starts from the master program \eqref{eq:lr} for $\relFF$ without VIs.
\begin{subequations}\label{eq:lr}
\begin{align}
\min \quad   & \sum_{F\in \CGfrag} \costF{F} x_{F} \label{eq:lr:obj}  \\
\mbox{s.t.}\quad &  \eqref{eq:ff:vehicle:limit}-\eqref{eq:ff:flow}, \eqref{eq:ff:time:all}, \eqref{eq:ff:td:uv:le}-\eqref{eq:ff:td:vu:ge}, \eqref{eq:ff:load:all} \nonumber \\
        & x_{F} \geq 0 & \forall F \in \CGfrag \\ 
        & p_{uv} \geq 0 & \forall \{u,v\} \in D, u < v
\end{align}
\end{subequations}

Initially, the set of fragments $\CGfrag$ consists of all fragments containing exactly two nodes from $\{0\}\cup \TDVis$, all fragments corresponding to vehicles routes starting and ending at the depot that visit exactly one task without temporal dependencies, and an artificial fragment with very high cost that satisfies all constraints and therefore ensures feasibility.
The set $\CGfrag$ is iteratively extended and VIs are added as cutting planes until the optimal solution to \relFFVI is found.
Within each pricing iteration, the dual values of the constraints are used to identify fragments starting at $v$ with minimal reduced costs for each task $v \in \TDVis \cup \{0\}$.
If no fragments with negative reduced costs are found, the row generation procedure separates and adds FSECs~\eqref{eq:FSEC}, TIFIs~\eqref{eq:InfPathV}, TDIFIs~\eqref{eq:InfPathTD}, and RCCs~\eqref{eq:RCC}.
FRCCs~\eqref{eq:FRCC} are not considered in this step due to their challenging separation problem and their impact on the pricing problem.
The column-and-row generation process terminates if neither negative reduced cost variables nor VIs have been added.

\subsubsection{Pricing Problem}\label{sss:solappr:lb:subproblem}

The pricing problem of the column generation algorithm is defined as 
\begin{equation}\label{eq:pricing_problem}
    \begin{array}{ll}
        \displaystyle \min_{\substack{
            u, v \in V_D \cup \{0\}, \\
            F \in \AllFrag_{u \to v}
        }} \Bigg\{ 
             & \sum_{(i,j)\in F} \rcArc{i}{j} 
              - \mathbf{1}_{u = 0}\,\dMaxVeh
            + \mathbf{1}_{u \in \TDVis}\, \dFlow{u}    \\
         & - (\durF{F} + \beta_{u} - \alpha_{v})\,\dTimeMTZ{uv} + \esF{v}{F}\,\dTimeLB{v}  - \lsF{u}{F}\,\dTimeUB{u} 
            - (\VehCap - q_{u} + \demF{F})\,\dLoadMTZ{uv}
            + \demF{F}\,\dLoadLB{v} + \demF{F}\,\dLoadUB{u} 
        \Bigg\}
    \end{array}
\end{equation}

where $\rcArc{i}{j} = c_{ij}  - \mathbf{1}_{i \in V} \dAss{i} - \mathbf{1}_{j\in \TDVis} \dFlow{j}$ denotes the reduced cost of arc $(i,j)\in F$, while $\dMaxVeh,\, \{\dAss{v}\}_{v \in V},\, \{\dFlow{u}\}_{v \in \TDVis},\, \{\dTimeMTZ{uv}\}_{u,v \in \TDVis},\, \{ \dTimeLB{v} \}_{v \in \TDVis},\, \{\dTimeUB{v}\}_{v \in \TDVis},\, \{\dLoadMTZ{uv}\}_{u,v \in \TDVis},\, \{\dLoadLB{u}\}_{u \in \TDVis},\, \{\dLoadUB{u}\}_{u \in \TDVis}$ are the dual variables of constraints \eqref{eq:ff:vehicle:limit}, \eqref{eq:ff:ass}, \eqref{eq:ff:flow}, \eqref{eq:ff:time:mtz}, \eqref{eq:ff:time:lb}, \eqref{eq:ff:time:ub}, \eqref{eq:ff:load:mtz}, \eqref{eq:ff:load:lb}, and \eqref{eq:ff:load:ub}, respectively. For notational convenience we set $\dTimeMTZ{0v}=\dTimeMTZ{v0}=\dTimeLB{0}=\dTimeUB{0}=\dLoadMTZ{0v}=\dLoadMTZ{v0}=\dLoadLB{0}=\dLoadUB{0}=0$, $\forall v \in \TDVis$. 
The completion cost of fragment $F$ ending at visit $v\in \TDVis$ is defined as 
\begin{equation*}
 \rcComplFrag{F}= - (\durF{F} + \beta_{u} - \alpha_{v})\,\dTimeMTZ{uv} + \esF{v}{F}\,\dTimeLB{v}  - \lsF{u}{F}\,\dTimeUB{u}             - (\VehCap - q_{u} + \demF{F})\,\dLoadMTZ{uv}
            + \demF{F}\,\dLoadLB{v} + \demF{F}\,\dLoadUB{u}.  
\end{equation*}

The completion cost $\rcComplFrag{F}$ is non-decreasing in $\durF{F}, \, \esF{v}{F}$, and $\demF{F}$, and non-increasing in $\lsF{u}{F}$ since $\dTimeMTZ{uv} \leq 0 , \, \dTimeLB{v} \geq 0, \, \dTimeUB{u} \geq 0, \, \dLoadMTZ{uv} \leq 0,  \, \dLoadLB{u} \geq 0$, and $\dLoadUB{u} \geq 0$ for all $u,v \in \TDVis \cup \{0\}$.

\subsubsection{Dynamic Programming}\label{sss:solappr:lb:dynamic:program}
In the following, we describe our dynamic programming based algorithm to solve the pricing problem~\eqref{eq:pricing_problem}. 
\paragraph{Label definition}
Let $f = (\flStart{f}=v_1, \dots, v_{\ell} = \flEnd{f})$ be a forward path that starts from $\flStart{f}$ no later than $\flLS{f}$, visits the set of tasks $\flSet{f} = \{v_1, \dots, v_{\ell}\}$, ends at $\flEnd{f}$ not earlier than $\flES{f}$, and has a duration of $\flDur{f}$. 
A forward path $f$ is a fragment if it satisfies the conditions given in \cref{def:frag}. 

In our dynamic programming procedure, we represent a forward path $f$ with label $L^{f} = (\flStart{f}, \flEnd{f},\flSet{f}, \flLoad{f}, \flDur{f}, \flES{f},\flLS{f},\flCost{f})$, where $\flCost{f}$ represents the reduced cost of $f$ as defined in \eqref{eq:pricing_problem}. Notice that if $f$ is not a fragment, i.e., $\flEnd{f} \notin V_D \cup \{0\}$, then $\flCost{f}$ does not include the terms associated with $\flDur{f}$, $\flES{f}$, $\flLS{f}$, and $\flLoad{f}$, since they can be exactly calculated only when the fragment is completed. 
A label $L^{f}$ is feasible if the following conditions are satisfied:

\[
\left\{
\begin{array}{l}
\flStart{f} \in  \TDVis \cup \{ 0 \} , \flEnd{f} \in N \\
\flSet{f} \subseteq N, \\
\flLoad{f} + q_{\flEnd{f}} \leq \VehCap, \\
\flLS{f} \in [\alpha_{\flStart{f}}, \beta_{\flStart{f}}], \\
\flES{f} \in [\alpha_{\flEnd{f}}, \beta_{\flEnd{f}}], \\
\flLS{f} + \flDur{f} \leq \beta_{\flEnd{f}}, \\
\flES{f} - \flDur{f} \geq \alpha_{\flStart{f}}, \\
\flDur{f} \in [\dmin{\flStart{f} \flEnd{f}}, \dmax{\flStart{f} \flEnd{f}}], \text{if } \flStart{f}, \flEnd{f} \in D.
\end{array}
\right.
\]
A feasible label $L^{f}$ represents an \emph{incomplete fragment} if $|\flSet{f}| = 1$, or if $|\flSet{f}| \geq 2$ and $\flEnd{f} \notin V_D \cup \{0\}$, and a \emph{fragment} otherwise.
 Fragments correspond to feasible solutions of the subproblem \eqref{eq:pricing_problem} and therefore labels representing them cannot be extended further. 

\paragraph{Label Initialization}
The labeling algorithm starts with a label $L^{f}$ for each task $v \in \TDVis \cup \{0\}$, where $f = (v)$ is a forward path visiting only $v$. 
The values of the label are initialized 
as $\flStart{f}=v, \flEnd{f}= v, \flSet{f}= \{v\}, \flLoad{f} =0, \flDur{f}=0, \flES{f}= \alpha_v, \flLS{f} = \beta_v$, and $\flCost{f} = - \mathbf{1}_{u = 0}\,\dMaxVeh + \mathbf{1}_{u \in \TDVis} \dFlow{u}$. 

\paragraph{Label extension}
A label $L^{f}$ representing an incomplete fragment can be extended by adding a new task $u \in (N \setminus \flSet{f}) \cup \{0\}$ to obtain a new, possibly incomplete fragment $g = (\flStart{f}, \dots, \flEnd{f}, u)$. 
  In line with \cref{prop:frag:derive:sched:cond}, the values of the new label are updated as follows:  
\begin{equation*}
\left\{
\begin{array}{l}
\flStart{g}= \flStart{f} \\ 
\flEnd{g} = u \\ 
\flSet{g} = \flSet{f} \cup \{u\} \\
\flLoad{g} = \flLoad{f} + q_{\flEnd{f}} \\ 
\flDur{g} = \max\{ \flDur{f} + d_{\flEnd{f}} + t_{\flEnd{f}u}, \alpha_{u} - \flLS{f}, \dmin{\flStart{f}u} \} \\  
\flES{g} = \max\{\flES{f} + d_{\flEnd{f}} + t_{\flEnd{f}u}, \alpha_u, \alpha_{\flStart{f}} + \dmin{\flStart{f}u} \} \\ 
\flLS{g} = \min\{ \flLS{f}, \beta_{u} - t_{\flEnd{f}u} - d_{\flEnd{f}} - \flDur{g},\beta_{u} - \dmin{\flStart{f}u}\} \\  
\flCost{g} = \flCost{f} + \rcArc{\flEnd{f}}{u} + \mathbf{1}_{u \in \TDVis\cup \{0\}} \rcComplFrag{g}
\end{array}
\right.
\end{equation*}
where $\dmin{\flStart{f}u} = 0$ if $\{\flStart{f},u\} \notin D$ and the completion cost $\rcComplFrag{g}$ are only considered when $u$ is a task with temporal dependency or the depot which completes the fragment, i.e., if $u\in V_D\cup \{0\}$.

\paragraph{Dominance Rule}
Given two labels $L^{f}$ and $L^{g}$, $f$ is said to dominate $g$ if  $\flStart{f} = \flStart{g}$, $\flEnd{f}=\flEnd{g}$, $\flSet{f} \subseteq \flSet{g}$ (and therefore $\flLoad{f} \leq \flLoad{g}$), $\flES{f} \leq \flES{g}$, $\flLS{f} \geq \flLS{g}$, $\flDur{f} \leq \flDur{g}$, and $\flCost{f} \leq \flCost{g}$. In this case, label $L^{g}$ can be discarded since all the feasible extensions of $g$ can be replaced by extensions of $f$ with no higher cost.
Since the label extension rules are monotone in all resources, the dominance conditions ensure that every feasible extension of $g$ is also a feasible extension of $f$ with resource values that are at least as favorable. Combined with the monotonicity of the completion costs and the condition $\flCost{f} \leq \flCost{g}$, this guarantees that $\flCost{f \oplus \hat{f}} \leq \flCost{g \oplus \hat{f}}$ for every feasible extension $\hat{f}$ of $g$, where  $\oplus$ denotes the  concatenation of two forwards paths.

\subsubsection{Acceleration Strategies}\label{sss:solappr:lb:acce}
The labeling algorithm is further accelerated by the following strategies.
\paragraph{$ng-$path relaxation} Instead of generating elementary fragments, we apply $ng-$path relaxation \citep{Baldacci2011} to achieve a tradeoff between bound quality and complexity of the pricing problem. Neighborhoods do not include any task with temporal dependencies as a cycle between those tasks cannot be part of any feasible fragment.

\paragraph{Improved Dominance Between Incomplete Fragments}
The dynamic programming algorithm generates many labels representing incomplete fragments. According to the updating and dominance rule described above for incomplete fragments $f$, the reduced costs associated with the scheduling possibilities and collected demand are not included in $\flCost{f}$ since they can be exactly calculated only when the fragment is completed.
 
Therefore, two incomplete fragments $f$ and $g$ may satisfy all dominance criteria in favor of $f$ except $\flCost{f} > \flCost{g}$. This prevents $f$ from dominating $g$, even when all completions of $f$ result in lower total reduced cost than those of $g$. In other words, the algorithm misses some dominance opportunities due to incomplete reduced cost information.

To address this issue and improve dominance detection, we use lower bounds on completion costs to improve dominance detection between incomplete fragments.
Specifically, for two incomplete fragments $f$ and $g$ where all dominance criteria except the reduced cost are satisfied in favor of $f$ (i.e., $\flCost{f} > \flCost{g}$), we can still dominate $g$ by $f$ if 
\begin{equation}\label{eq:cond:lb:compl:dom}
    \flCost{f} \leq \flCost{g} + \rcBoundDom{f}{g},
\end{equation}
where $\rcBoundDom{f}{g}$ underestimates the completion cost difference for any feasible completions of $g$ and $f$, i.e.,
\begin{equation}\label{eq:cond:lb:compl:cost}
    \rcBoundDom{f}{g} \leq \min_{\hat{f} : g \oplus \hat{f} \in \AllFrag} \rcComplFrag{g \oplus \hat{f}} - \rcComplFrag{f \oplus \hat{f}}.
\end{equation}

\begin{restatable}{theorem}{dpdominanceboundcomplcost}\label{thm:dp:dominance:bound:compl:cost}

Consider two forward paths $f$ and $g$ that satisfy all dominance criteria in favor of $f$ except that $\flCost{f} > \flCost{g}$, and the set $V_{\dmin{}}^{g}$ of tasks with temporal dependencies to starting task $\flStart{g}$ of $g$ to which $g$ can be extended, i.e., $V_{\dmin{}}^{g} = \{ u : \{\flStart{g},u\} \in D \land  \flES{g} + d_{\flEnd{g}} + t_{\flEnd{g}u} \leq \beta_u \land \alpha_{\flStart{g}} + \dmin{\flStart{g}u} \leq \beta_u  \land \max\{ \flDur{g} + d_{\flEnd{g}} + t_{\flEnd{g}u}, \alpha_{u} - \flLS{g}\} \leq \dmax{\flStart{g}u} \land \flLoad{g} + q_{u} \leq \VehCap\}$. A fragment resulting from extending $g$ to $u \in V_{\dmin{}}^{g}$ cannot start later than $\min\{\beta_{\flStart{g}}, \beta_{u} - \dmin{\flStart{g}u}\}$ and, consequently, $\flLSDMin{g} = \min\{\beta_{\flStart{g}}, \min_{u \in V_{\dmin{}}^{g}}\{\beta_{u} - \dmin{\flStart{g}u}\}\}$ is the smallest latest starting time that can be implied be a temporal dependency. 
Then, \begin{equation}\label{eq:def:lb:compl:cost}
 \rcBoundDom{f}{g} = 
   \left( \min\{ \flLS{f}, (\flLS{g} + \flDur{g}) - \flDur{f}, \max \{ \flLS{g},\flLSDMin{g} \} \} - \flLS{g} \right) \dTimeUB{\flStart{f}} + (\flLoad{g} - \flLoad{f})\dLoadUB{\flStart{f}},
\end{equation}
is a valid lower bound on the difference in completion cost between $f$ and $g$, i.e., satisfying condition \eqref{eq:cond:lb:compl:cost}, and can be used in dominance criterion \eqref{eq:cond:lb:compl:dom}. 

\end{restatable}

\subsubsection{Row generation}\label{sss:solappr:lb:vi}
The row generation procedure separates and adds violated FSECs~\eqref{eq:FSEC}, TIFIs~\eqref{eq:InfPathV}, TDIFIs~\eqref{eq:InfPathTD}, and RCCs~\eqref{eq:RCC} at the end of each iteration of the column generation algorithm. 
Afterwards, the column generation procedure is repeated in case at least one violated inequality is added. 

In the following, we describe the separation procedures for these VIs whose dual values are considered in the labeling algorithm by adjusting arc costs and completion costs accordingly. 
These adjustments as well as an acceleration technique that accounts for them are detailed in \cref{app:lab:alg:vi}. 
\paragraph{FSECs~\eqref{eq:FSEC}}
are separated using two heuristic separation procedures since exact separation is NP-Hard (see \cref{sec:fsec}). 
Both procedures use 
the support graph $G[\xsol]= (\{0\} \cup \TDVis,A[\xsol])$ implied from the current solution $\xsol$, i.e., $A[\xsol] = \{(u,v) : u,v \in \{0\} \cup \TDVis \land \exists F \in \CGfragFromTo{u}{v} : \xsol_{F} > 0\}$ with arc weights $w_{uv} = \sum_{F \in \CGfragFromTo{u}{v}} \xsol_{F}$, $\forall (u,v) \in A[\xsol]$.
The first procedure identifies subtours in $G[\xsol]$ that do not correspond to feasible routes by computing a maximum flow from $0$ to each $v\in \TDVis$. If an infeasible subtour for set $S\subseteq \TDVis$ is identified, the minimum number of vehicles $\FSECVehMin{S}$ required to serve $S$ is computed by solving an adjusted version of the MILP formulation introduced in \cref{app:af}, only considering tasks $S \cup \{0\}$ and adjusting the objective to minimize the number of used vehicles, and the associated cut is added.
The second procedure enumerates all sets $S \subseteq \TDVis$ of cardinality at most $k_\mathrm{max}$. If $\sum_{(u,v) \in A[\xsol]: u,v \in S} w_{uv} > 1$, $\FSECVehMin{S}$ is computed by increasing the number of vehicles from $1$ to $|S|$ until a feasible solution is found. 
The feasibility of a specific number of vehicles is checked by enumerating all corresponding route possibilities. The feasibility of a route combination is verified by first checking the individual routes and subsequently enforcing the temporal dependency constraints in an iterative manner (see \cref{app:sep:FSEC:alg}).
\paragraph{TIFIs~\eqref{eq:InfPathV}}
are separated by simply checking for violations at each time point corresponding to an earliest completion time of fragments used in the current solution, cf.\ \cref{thm:InfPathV:num} which shows that these TIFIs dominate all others.
At most one TIFI is added for each task, thereby choosing one with maximal violation in case multiple TIFIs are violated for the same task.

\paragraph{TDIFIs~\eqref{eq:InfPathTD}} are also separated by inspection. 
For every temporal dependency $\{u,v\} \in D$, the separation routine checks for violations at the earliest completion times and latest starting times of the fragments in the current solution, in line with \cref{thm:InfPathTD:num}. If multiple violations occur for a single temporal dependency pair, one maximally violated inequality among inequalities \eqref{eq:InfPathTD:uv:min}-\eqref{eq:InfPathTD:vu:max} is added. 

\paragraph{RCCs \eqref{eq:RCC}} are separated and added in undirected form using the procedures from \cite{Lysgaard2004}.

\subsection{Branch-and-cut}\label{ss:solappr:MIP}
We use branch-and-cut to obtain an initial upper bound by solving $\FFVIm[\CGfrag]$ in Step~\ref{sol:proc:step:init:ub} and to obtain an improved, possibly optimal solution by solving $\FFVIm[\ENfrag\cup \IncSol]$ in Step~\ref{sol:proc:step:MIP}. In both of these steps, TIFIs~\eqref{eq:InfPathV}, TDIFIs~\eqref{eq:InfPathTD}, and RCCs~\eqref{eq:RCC} are separated at fractional solutions in this order and using the separation routines described in \cref{sss:solappr:lb:vi}. Thereby, an inequality is only added if its violation exceeds a given threshold and if the pre-defined maximum number of inequalities of its type is not yet reached. Furthermore, in each cutting-plane iteration inequalities of a particular type are only separated if no violated inequalities of earlier considered types have been added. 
Notice that FSECs~\eqref{eq:FSEC} are not considered since separating them consumes too much computation time and that RCCs~\eqref{eq:RCC} are added instead of FRCCs~\eqref{eq:FRCC} since considering the latter in this phase did not pay off in preliminary experiments, likely due to the fact that they lead to dense cuts involving a large number of variables for larger subsets of tasks. 

In addition to the different sets of fragments considered, Steps~\ref{sol:proc:step:init:ub} and \ref{sol:proc:step:MIP} differ in the following two aspects: 
\begin{inparaenum}[(i)]
    \item a time limit $\CGtimeMIP$ is applied in Step~\ref{sol:proc:step:init:ub} to avoid spending too much time on finding an initial upper bound, and
    \item all RCCs~\eqref{eq:RCC} added during row generation in Step~\ref{sol:proc:step:lb} are replaced by their stronger representation as FRCCs~\eqref{eq:FRCC} before starting the branch-and-cut in Step~\ref{sol:proc:step:init:ub}.
\end{inparaenum}
The last step aims to improve the performance of the branch-and-cut in both Steps~\ref{sol:proc:step:init:ub} and \ref{sol:proc:step:MIP} by strengthening the initial lower bound.

\subsection{Fragment Enumeration}\label{ss:solappr:frag:enum}

In this step, we enumerate the set of fragments $\ENfrag$ whose reduced costs do not exceed the difference between the current guess for the optimal solution value $\UBguess$ and the lower bound $\VAL[\relFFVI[\CGfrag]]$ identified via column-and-row generation in Step \ref{sol:proc:step:lb}. It is well known, that set $\ENfrag$ contains all fragments included in solutions whose objective value is within $[\VAL[\relFFVI[\CGfrag]],\UBguess]$. 

To enumerate set $\ENfrag$, all fragments that correspond to solutions to the pricing problem~\eqref{eq:pricing_problem} whose reduced costs w.r.t.\ the dual values obtained in the last pricing iteration do not exceed  $\UBguess-\VAL[\relFFVI[\CGfrag]]$ need to be identified. To ensure enumerating all relevant columns, the dominance rules described in \cref{ss:solappr:init:lb} cannot be used. In the enumeration step, a fragment $F$ only dominates a fragment $F'$ if $F'$ can be replaced by $F$ in every feasible solution without increasing the solution cost.
Thus, the dynamic program as described in \cref{sss:solappr:lb:dynamic:program} can be applied to solve the enumeration problem with two modifications: 
\begin{inparaenum}[(i)]
    \item the dominance rules compare total travel cost of forward paths instead of their reduced cost and
    \item the $ng$-path relaxation is not used to ensure elementarity.    
\end{inparaenum}
The enumeration step applies a backward version of the dynamic programming algorithm described in \cref{ss:solappr:init:lb} which is accelerated through the use of completion bounds \citep{Baldacci_Mingozzi_Roberti_2012} obtained during the last execution of the forward labeling algorithm. 
The completion bounds, thereby, account for the the minimal costs that are associated to $\esF{\cdot}{F}, \; \lsF{\cdot}{F}$, and $\demF{F}$ when merging two labels. A second acceleration strategy is that dominance is established by identifying beneficial 2-opt and 3-opt moves instead of pairwise label comparisons. To avoid excessive memory usage, the procedure and the overall solution algorithm stops if more than $\maxfragenum$ relevant fragments are identified.

After the enumeration, the following fragment elimination procedure is used to reduce the number of considered fragments. It is based on the observation that a fragment can only contribute to a solution with an objective value in $[\VAL[\relFFVI[\CGfrag]], \UBguess]$ if it is contained within at least one route whose reduced cost does not exceed $\UBguess - \VAL[\relFFVI[\CGfrag]]$. Following the idea of formulation leveraging by \cite{Sippel2024} this observation is exploited by computing a lower bound on the minimum reduced costs of any route that contains a given fragment. For each fragment $F=(u, \dots, v)\in \ENfrag$, such a lower bound is obtained by considering its reduced cost plus the minimum sum of reduced costs of any two fragments $F'\in \ENfragTo{u}$ and $F''\in \ENfragFrom{v}$ while considering only those fragments that can be scheduled before and after $F$, respectively, i.e., such that $\esF{u}{F'} \leq \lsF{u}{F}$ and $\lsF{v}{F''} \geq \esF{v}{F}$. 
The procedure removes every fragment from $\ENfrag$ whose bound exceeds $\UBguess - \VAL[\relFFVI[\CGfrag]]$, and is accelerated by considering disjunctive $ng$-neighborhoods at linking tasks.
Finally, $\relFFVIm[\ENfrag \cup \IncSol]$ is solved resulting in new dual values and a solution value $\VAL[\relFFVIm[\ENfrag \cup \IncSol]]$ that is potentially higher than $\VAL[\relFFVI[\CGfrag]]$ due to the replacement of RCC \eqref{eq:RCC} by FRCC \eqref{eq:FRCC}. The reduced costs for each fragment are recalculated and the fragments in $\ENfrag$ for which these updated reduced costs exceed the bound $\UBguess - \VAL[\relFFVIm[\ENfrag \cup \IncSol]]$ are discarded.

\section{Computational study}\label{sec:compres}
This section studies the performance of \solalg and compares it to two benchmark methods. The first such method is a re-implementation of the branch-and-price approach of \citet{Dohn2011} based on time window branching, which has been used in a variety of studies \citep{DOHN20091145,RASMUSSEN2012598,Luo2016,LIN2021102177,Kling2025} since it can handle a broad range of temporal dependency constraints. Since this approach is not applicable to non-overlapping temporal dependencies, we also compare \solalg to solving an arc-based formulation with a general purpose solver. A more detailed description of these two benchmark methods is provided in \cref{app:bench:appr}.

Below, we summarize main parameters used in our experiments before providing details about the used benchmark instances in \cref{ss:bench:inst}. The results of our computational study which focuses on 
\begin{inparaenum}[(i)]
    \item analyzing the impact of different components of \solalg,
    \item comparing its performance against the two benchmark methods, and
    \item assessing its performance on larger instances whose size is prohibiting for applying existing methods
\end{inparaenum}
are discussed in \cref{ss:comp:vi,ss:comp:sol:alg,ss:perf:size:75}.
All algorithms were implemented in the programming language \textit{Julia} and executed using Gurobi 12. The computation study was performed using the \textit{AMD Rome 7H12} nodes of the Dutch National Supercomputer Snellius whereby a time limit of two hours was imposed for solving a single instance. 
Based on the results of preliminary experiments, \solalg has been further configured as follows:
FSECs for each pair of temporal dependency tasks and the set $\TDVis$ as well a constraint enforcing a lower bound on the number of used vehicles are initially included in Step~\ref{sol:proc:step:lb}.
At most $100$ fragments with negative reduced costs are added in each column generation iteration where the size of the $ng$-path neighborhood is set to $10$. The row generation procedure enumerates subsets of size at most $k_\mathrm{max}=5$ when identifying violated FSECs, and adds at most $200$ violated RCCs in total. 
A time limit of $\CGtimeMIP=100$ is applied to the branch-and-cut in Step~\ref{sol:proc:step:init:ub} which replaces only those RCCs added during row generation by their stronger representation as FRCCs that include at most $\lceil 0.25 \cdot |V| \rceil$ tasks and uses $\UBinitstep=0.05$ to calculate the initial guess of the optimal solution value. The fragment limit in the enumeration step is set to 
$\maxfragenum=20\,000\,000$ and the minimum violations for which inequalities are added during the branch-and-cut procedures in Steps~\ref{sol:proc:step:init:ub} and \ref{sol:proc:step:MIP} are set to $0.25$, $0.25$, and $0.05$ for TIFIs, TDIFIs, and RCCs, respectively, with a limit of at most 1\,000 cuts of each type. Finally, $\UBstep=0.05$ in Step~\ref{sol:proc:recursion}.

\subsection{Benchmark instances}\label{ss:bench:inst}
We created a set of benchmark instances based on the 56 instances of \citet{Solomon87} originally proposed for the capacitated VRP with tighter (type 1) and looser time windows and capacity restrictions (type 2). For each original instance with 100 tasks, we first create two base instances by considering the first 50 and 75 tasks only. From each of the resulting $2\cdot 56=112$ base instances, we create $6$ instances each considering one specific type of temporal dependencies following a similar procedure as \citet{Dohn2011} because the original instances could not be retrieved. The considered types are synchronization, minimum precedence, maximum precedence, minimum/maximum precedence, overlap, and non-overlap and we will refer to the instances with these types as syn, min, max, minmax, overlap, and non-overlap, respectively. 
Temporal dependencies are generated by iteratively selecting a random pair of tasks without implicit or explicit temporal dependencies between them and creating a restrictive temporal dependency if possible. Hereby, a restrictive temporal dependency means that it limits the scheduling possibilities 
of at least one of the two tasks compared to the current set of temporal dependencies, without imposing infeasible timing restrictions. The minimum and/or maximum difference for precedence constraints is randomly drawn from the set of values that restrict the possible starting times of a least one of these 
tasks.
The procedure stops when there is no node pair left for which a restrictive temporal dependency can be added. By selecting the first $\lceil \reltempdep \cdot |V|\rceil$ temporal dependencies from each of these instances for $\reltempdep\in \{0.05, 0.15, 0.25\}$, we obtain our final benchmark set consisting of $3 \cdot 56 = 168$ instances per temporal dependency type and thus $6 \cdot 3 \cdot 56=1\,008$ instances in total for each considered number of tasks $|V|\in \{50,75\}$. 
Lastly, since each task has integer duration and time window bounds, travel times are 
rounded up to the nearest integer resulting in integer solution values. 

\subsection{Impact of valid inequalities}\label{ss:comp:vi}
We first assess the contributions of the valid inequalities on the performance of \solalg using the 1\,008 instances with 50 tasks. \cref{tab:conf:key:char} reports key characteristics for seven configurations of \solalg considering different types of valid inequalities.
Next to the baseline configuration \solalg described in \cref{sec:solappr}, we also consider \solalgNoBC which disables cutting planes in the branch-and-cut in Steps~\ref{sol:proc:step:init:ub} and \ref{sol:proc:step:MIP} and variants \solalgNoBCNoCut{\texttt{type}} that additionally exclude either a specific type of inequalities $\texttt{type} \in \{\text{FSEC}, \text{TIFI}, \text{TDIFI}, \text{RCC}\}$ or all of them (\solalgNoBCNoCut{VI}) when computing the initial lower bound in Step~\ref{sol:proc:step:lb}. 
\cref{tab:conf:key:char} reports the numbers of instances solved to optimality (opt), the average and maximum gaps in percentage (gap, gap$_\mathrm{max}$), and the average computation times in seconds (time), where runtimes are set to 7200 seconds if \solalg terminates early because more than $\maxfragenum$ fragments are identified. Moreover, time$^*$ denotes the average computation time on instances solved to optimality by all considered variants, and $\widehat{\text{gap}}$ denotes the average optimality gap on instances where at least one configuration fails to prove optimality.

\begin{table}
\TABLE
{Impact of valid inequalities on the performance of \solalg. \label{tab:conf:key:char}}
{\begin{tabular}{rcccccc}
\toprule
Configuration & opt & time & gap & gap$_\mathrm{max}$ & time$^*$ & $\widehat{\text{gap}}$ \\
\midrule
   \solalg                  & \textbf{837}  & \textbf{1393} & 1.42          & 25.93              & 100           & 6.34 \\ 
  \solalgNoBC               & \textbf{837}  & 1415          & \textbf{1.41} & \textbf{23.24}     & 104           & \textbf{6.31} \\
  \solalgNoBCNoCut{FSEC}    & 820           & 1523          & 1.64          & 29.26             & 146           & 7.34 \\
  \solalgNoBCNoCut{TIFI}    & 810           & 1586          & 1.72          & 25.58             & 171           & 7.71 \\
  \solalgNoBCNoCut{TDIFI}   & 830           & 1443          & 1.49          & 26.20             & 118           & 6.68 \\
  \solalgNoBCNoCut{RCC}     & \textbf{837}  & 1394          & 1.42          & 24.59             & \textbf{89}   & 6.36 \\
  \solalgNoBCNoCut{VI}      & 790           & 1815          & 2.33          & 45.33             & 295           & 10.44 \\
\bottomrule
\end{tabular}}{}
\end{table}

The difference between \solalgNoBC and \solalgNoBCNoCut{VI}, observed from \cref{tab:conf:key:char}, clearly shows the positive impact of using the valid inequalities while identifying the lower bound in Step~\ref{sol:proc:step:lb} on the overall performance. Separating the valid inequalities during row generation increases the number of instances solved to optimality from 790 to 837, reduces the average optimality gap from 2.33\% to 1.42\%, and decreases the average runtime from 1\,815s to 1\,415s. The results for variants excluding individual types of inequalities reveal that FSECs and TIFIs appear to have the largest positive effect, with TIFIs contributing the most. In contrast, the impact of RCCs seems limited, likely because of the fact that the introduced temporal dependencies reduce the restrictiveness of the capacity constraints. Their inclusion slightly reduces the average and maximum optimality gap,
but increases the average solution time. The latter observation also holds on the subset of instances solved to optimality by all configurations.

In contrast to their significant impact during Step~\ref{sol:proc:step:lb}, the impact of separating inequalities during branch-and-cut seems very limited (cmp.\ variants \solalg and \solalgNoBC) and their inclusion does not yield better results. Although the average solution time decreases, the optimality gaps slightly increase.
Further insights presented in \cref{app:ins:FSA} confirm the positive contribution of considering valid inequalities in Step~\ref{sol:proc:step:lb} by showing that each type increases the lower bound and improves the algorithm performance, with FSECs and TIFIs having the strongest effect. 
We note that these results also show the effectiveness of the fragment reduction procedures that remove, on average, approximately 33\% of the fragments.

Overall, \solalgNoBC seems the best-performing variant as it solves the largest number of instances to optimality and achieves the smallest average and maximum optimality gaps. \solalgNoBC will, therefore, be used in the remainder of this computational study. 

\subsection{Comparison to benchmark methods}\label{ss:comp:sol:alg}

This section compares the performances of \solalgNoBC to our re-implementation of the branch-and-price method by \citet{Dohn2011} (\dohnalg) and an arc-based formulation solved by Gurobi (\mtzalg) using the 1\,008 instances with 50 tasks. Thereby, we focus on the impact of the number of temporal dependencies and their type.

\subsubsection{Number of temporal dependencies}
We first analyze the impact of the number of temporal dependencies $\lceil \reltempdep \cdot |V|\rceil$ where $\reltempdep\in \{0.05, 0.15, 0.25\}$. \cref{fig:comp:dohn:gap:time,fig:comp:mtz:gap:time} show corresponding optimality gaps and solution times of \solalgNoBC and \dohnalg as well as \solalgNoBC and \mtzalg, respectively. 
Since \dohnalg is not applicable to non-overlapping temporal dependencies, those instances are excluded in \cref{fig:comp:dohn:gap:time}. 

\begin{figure}
     \FIGURE
     {\includegraphics[width=1.0\textwidth]{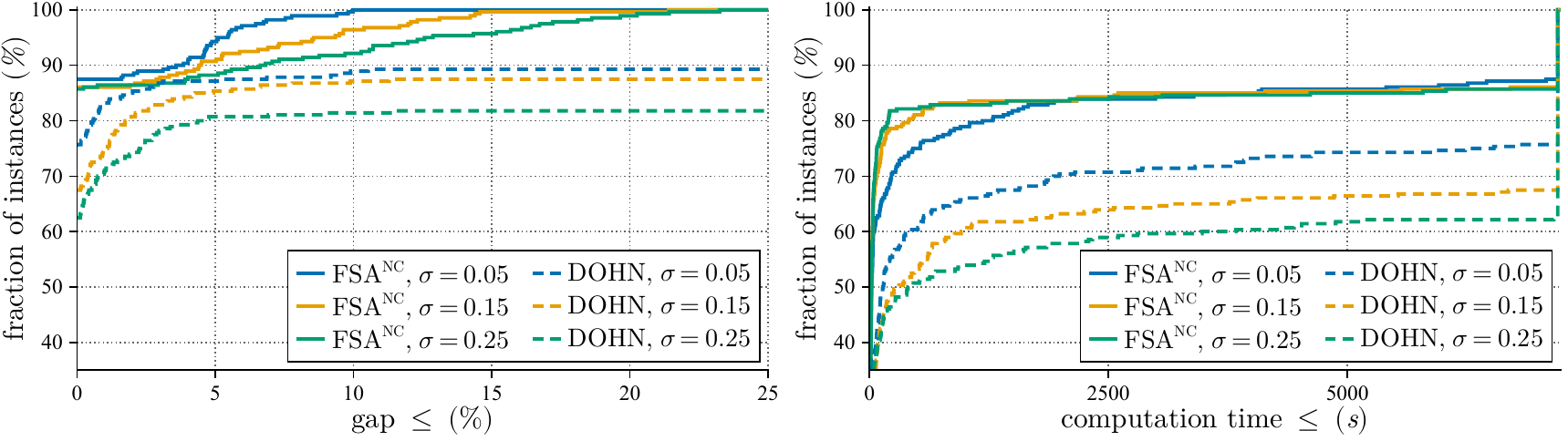}} 
{Optimality gaps and runtimes of \solalgNoBC and \dohnalg on instances with 50 tasks without non-overlapping temporal dependencies for various amounts of temporal dependencies. Optimality gaps are cutoff at 25\%. \label{fig:comp:dohn:gap:time}} 
{}
\end{figure}
\begin{figure}
     \FIGURE
     {\includegraphics[width=1.0\textwidth]{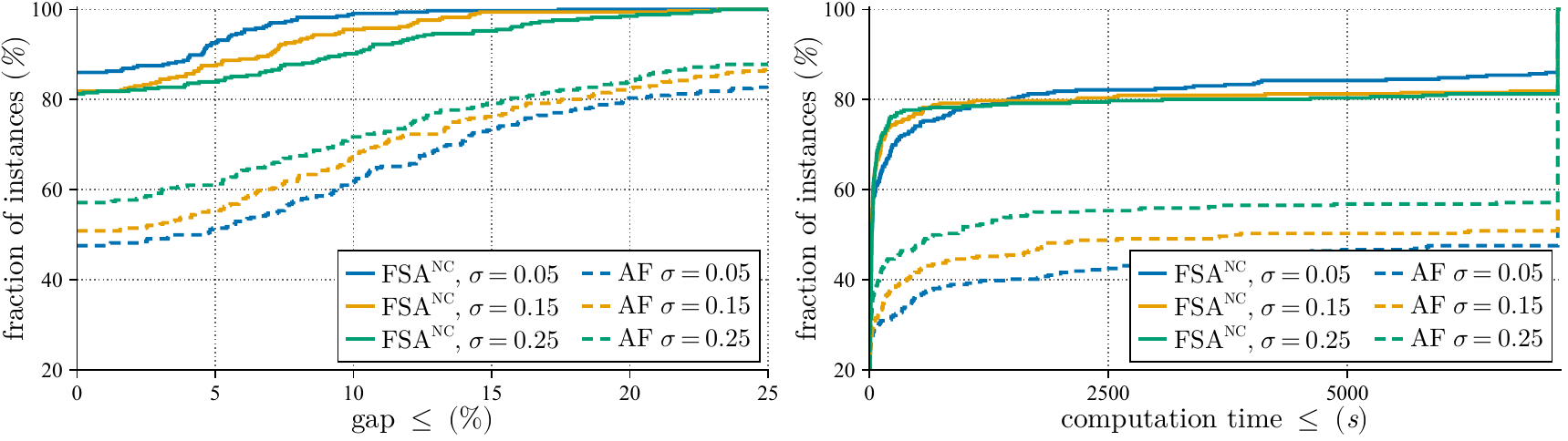}} 
{Optimality gaps and runtimes of \solalgNoBC and \mtzalg on instances with 50 tasks for various amounts of temporal dependencies. Optimality gaps are cutoff at 25\%. 
\label{fig:comp:mtz:gap:time}}
{}
\end{figure}

These results show that \solalgNoBC significantly outperforms \dohnalg and \mtzalg. It solves more instances to optimality, does so in a smaller amount of time, and achieves much lower optimality gaps for instances that could not be solved within the time limit. These observations hold consistently across all three amounts of temporal dependencies. We also observe a significant impact of the number of temporal dependencies on the performances of \dohnalg and \mtzalg. Interestingly, the performance of \dohnalg decreases with an increasing number of temporal dependencies, while the opposite is true for \mtzalg. For \solalgNoBC we observe that the number of solved instances slightly decreases, the optimality gaps on unsolved instances slightly increase, but surprisingly the computation times slightly decrease with an increasing amount of temporal dependencies. Overall, its performance is more robust than the one of the two benchmark methods. This highlights one main benefit of using the proposed fragment based formulation and the price-cut-and-enumerate algorithm naturally balancing the number of fragment variables, time spent in column generation and enumeration, and the strength of the formulation. Instances with many temporal dependencies lead to less effort in the pricing and enumeration step due to fewer and shorter fragments, but weaker bounds while the opposite is true for instances with few temporal dependencies in which case the fragment-based formulation is closer to a route model with more variables but tighter lower bounds.

\subsubsection{Type of temporal dependencies}

We now analyze the performance of the three considered methods for different types of temporal dependencies. 
\cref{tab:key:char:all:meth} reports numbers of instances solved to optimality ($\#_\mathrm{opt}$), average computation times in seconds (time), and average and maximum gaps in percentage (gap, gap$_{\mathrm{max}}$). The first three statistics are also reported on the subset of instances for which each method found a feasible solution, using the notation $\widetilde{\cdot}$, along with the number of instances in this subset ($\widetilde{\#}$). Results on this smaller set of instances are included due to a major limitation of \dohnalg which was not able to find a feasible solution for 116 instances, such that its average and maximum gaps could not be computed. These cases are spread across all types of temporal dependencies with the largest number of 39 cases arising for overlapping constraints. In contrast, \solalgNoBC and \mtzalg obtained valid lower and upper bounds for every instance.
All results are provided for each of the six temporal dependency types considered as well as aggregated over all types (all) and all types except non-overlapping one (all-nonov) that cannot be considered by \dohnalg.

\begin{table}
\TABLE
{Performance of \solalgNoBC, \dohnalg, and \mtzalg on instances with 50 tasks. \label{tab:key:char:all:meth}}
{\begin{tabular}{cccccccccc}
\toprule
Temp. Dep. & Method & \#$_\mathrm{opt}$ &  time & gap & gap$_\mathrm{max}$  & $\widetilde{\vphantom{\text{time}}\#}$ & 
 $\widetilde{\vphantom{\text{time}}\#}$$_\mathrm{opt}$& $\widetilde{\text{time}}$ & $\widetilde{\vphantom{\text{time}}\text{gap}}$   \\ 
 \midrule 
  \multirow{3}{*}{syn} & $\solalgNoBC$ & \textbf{141} &  \textbf{1339} &  \textbf{1.38} &  \textbf{23.24} &  \multirow{3}{*}{141} &  \textbf{131} &  \textbf{693} & 0.44 \\
   & $\dohnalg$ & 120 & 2353 & - & - &  & 120 & 1425 &  \textbf{0.36} \\
   & $\mtzalg$ & 88 & 3681 & 11.34 & 72.33 &  & 85 & 3148 & 8.82 \\
   \cmidrule(lr){1-2} \cmidrule(lr){3-6} \cmidrule(lr){7-10}
  \multirow{3}{*}{min} & $\solalgNoBC$ &  \textbf{158} &  \textbf{718} &  \textbf{0.29} &  \textbf{7.31} & \multirow{3}{*}{162} &  \textbf{155} &  \textbf{595} &  \textbf{0.20} \\
   & $\dohnalg$ & 134 & 1867 & - & - &  & 134 & 1669 & 0.24 \\
   & $\mtzalg$ & 100 & 3200 & 7.03 & 56.37 & & 100 & 3052 & 6.85 \\
 \cmidrule(lr){1-2} \cmidrule(lr){3-6} \cmidrule(lr){7-10}
  \multirow{3}{*}{max} & $\solalgNoBC$ &  \textbf{130} &  \textbf{1863} &  \textbf{2.01} &  \textbf{21.42} & \multirow{3}{*}{144} &  \textbf{123} &  \textbf{1307} & 1.17 \\
   & $\dohnalg$ & 110 & 2779 & - & - &  & 110 & 2042 &  \textbf{0.63} \\
   & $\mtzalg$ & 85 & 3820 & 11.20 & 69.82 &  & 81 & 3455 & 9.04 \\ 
    \cmidrule(lr){1-2} \cmidrule(lr){3-6} \cmidrule(lr){7-10}
  \multirow{3}{*}{minmax} & $\solalgNoBC$ &  \textbf{158} &  \textbf{536} &  \textbf{0.25} & \textbf{6.85} & \multirow{3}{*}{148} &  \textbf{143} &  \textbf{328} &  \textbf{0.14} \\
   & $\dohnalg$ & 112 & 2706 & - & - &  & 112 & 2098 & 0.41 \\
   & $\mtzalg$ & 96 & 3289 & 7.04 & 56.44 &  & 92 & 2953 & 6.37 \\ 
   \cmidrule(lr){1-2} \cmidrule(lr){3-6} \cmidrule(lr){7-10}
  \multirow{3}{*}{overlap} & $\solalgNoBC$ &  \textbf{139} &  \textbf{1404} &  \textbf{1.50} &  \textbf{20.25} & \multirow{3}{*}{129} &  \textbf{119} &  \textbf{659} &  \textbf{0.40} \\
   & $\dohnalg$ & 100 & 3241 & - & - &  & 100 & 2045 &  0.60 \\
   & $\mtzalg$ & 85 & 3761 & 11.44 & 68.20 &  & 77 & 3159 & 8.56 \\ 
 \cmidrule(lr){1-2} \cmidrule(lr){3-6} \cmidrule(lr){7-10}
  \multirow{2}{*}{non-overlap} & $\solalgNoBC$                           & \textbf{111}   & \textbf{2627} & \textbf{3.03}      & \textbf{23.20}                                       & \multirow{2}{*}{168}     & \textbf{111} & \textbf{2627} & \textbf{3.03}   \\
   & $\mtzalg$                                                      & 69    & 4397 & 12.30     & 62.30                                       &      & 69 & 4397 & 12.30  \\ \midrule 
  \multirow{3}{*}{\shortstack{all-nonov}} &  $\solalgNoBC$  & \textbf{726} &  \textbf{1172} &  \textbf{1.09} &  \textbf{23.24} & \multirow{3}{*}{724} &  \textbf{671} &  \textbf{713} & 0.46 \\
   & $\dohnalg$ & 576 & 2589 & - & - &  & 576 & 1850 &  \textbf{0.44} \\
   & $\mtzalg$ & 454 & 3550 & 9.61 & 72.33 &  & 435 & 3150 & 7.87 \\
 \cmidrule(lr){1-2} \cmidrule(lr){3-6} \cmidrule(lr){7-10}
 \multirow{2}{*}{all} & $\solalgNoBC$                               & \textbf{837}   & \textbf{1415} & \textbf{1.41}      & \textbf{23.24}                                       & \multirow{2}{*}{1008}      & \textbf{837} & \textbf{1415} & \textbf{1.41}  \\
   & $\mtzalg$                                                      & 523   & 3691 & 10.06     & 72.33                                       &      & 523 & 3691 & 10.06  \\ 
\bottomrule
\end{tabular}}{}
\end{table}

We first consider the aggregated results on instances with and without non-overlapping constraints (i.e., rows all and all-nonov) which confirm the superior performance of \solalgNoBC over \dohnalg and \mtzalg. On the largest set of instances that can be considered by all three approaches (i.e., the subset of instances excluding non-overlapping constraints) \solalgNoBC solves 726 out of the 840 instances to optimality, while \dohnalg only solves 576 instances and \mtzalg 454 instances. Furthermore, average and maximum optimality gaps of \solalgNoBC are much smaller than those of \mtzalg with values of 1.09\% compared to 9.61\% and 23.24\% compared to 72.33\%, respectively. As noted before, these gaps are not available for \dohnalg which fails to find a valid upper bound in 116 cases. In addition to solving more instances and achieving substantially lower optimality gaps, \solalgNoBC is on average more than two times as fast as \dohnalg and approximately three times as fast as the \mtzalg.
For the subset of instances where \dohnalg does find a feasible solution, \solalgNoBC again solves more instances to optimality than \dohnalg (671 vs.\ 576), and its average computation time is much smaller (713s vs.\ 1\,850s) while their optimality gaps are similar. The latter observation can be attributed to the quality of the linear relaxation bounds of \dohnalg which is based on a route rather than fragment-based formulation. 

The main conclusions drawn so far carry over to each individual type of temporal dependencies. \solalgNoBC consistently outperforms \dohnalg in terms of numbers of instances solved to optimality and with respect to average computing times. Since \mtzalg has the worst performance on all considered metrics and for all temporal dependencies we will mainly focus on the comparison between \solalgNoBC and \dohnalg in the following. We observe that \dohnalg struggles in particular with overlapping temporal dependencies where it can only solve 100 out of the 168 instances and fails to find any solution for 39 instances, while \solalgNoBC solves 139 of these instances and achieves an average gap of 1.5\%. The advantages of \solalgNoBC over \dohnalg are also particularly pronounced for minmax precedence, where \solalgNoBC solves 158 instances to optimality with an average gap of 0.25\%, while \dohnalg only solves 112 instances and fails to find a feasible solution for 20 instances. For these two types of temporal dependencies, the optimality gaps of \solalgNoBC are smaller even when considering only the subset of instances for which \dohnalg finds solutions. The advantages of \solalgNoBC in terms of numbers of solved instances, average runtimes, and consistently finding lower and upper bounds are also present for the remaining types of temporal dependencies, i.e., syn, min, max. No clear picture emerges, however, for these temporal dependencies when comparing the optimality gaps of \solalgNoBC and \dohnalg on the instances for which \dohnalg finds feasible solutions. Indeed, these gaps are slightly in favor of \solalgNoBC for minimum precedence, and (slightly) in favor of \dohnalg for synchronization and maximum precedence.

\smallskip
The results in \cref{tab:key:char:all:meth} also indicate that non-overlapping constraints, which have not been considered in the existing literature, seem the most challenging type. While \dohnalg cannot consider them at all, \solalgNoBC and \mtzalg solve fewer instances than for other types of temporal dependencies, require more time, and terminate with larger optimality gaps. Nevertheless, \solalgNoBC still solves the majority of the 168 instances and achieves an average solution gap of 3.03\% being clearly superior to the only competitor \mtzalg which solves only 69 instances and whose average optimality gaps are 9.61\%. 
Overall, the results in this section demonstrate that \solalgNoBC clearly outperforms both considered alternatives (\dohnalg and \mtzalg) on the considered set of benchmark instances independent of the type of temporal dependencies considered and their number.  
%

\subsection{Performance on large instances} \label{ss:perf:size:75} 
In this section, we discuss the performance of \solalgNoBC on instances with $75$ tasks for which the benchmark methods fail to produce meaningful results. For these experiments, the size of the $ng$-path neighborhood was reduced to 8. 
\cref{fig:ff:75:gap:time} shows optimality gaps and runtimes of \solalgNoBC for different amounts of temporal dependencies, while \cref{tab:meth:FF:key:char:75} summarizes its performance for different types of temporal dependencies.  

\begin{figure}
     \FIGURE
     {\includegraphics[width=1.0\textwidth]{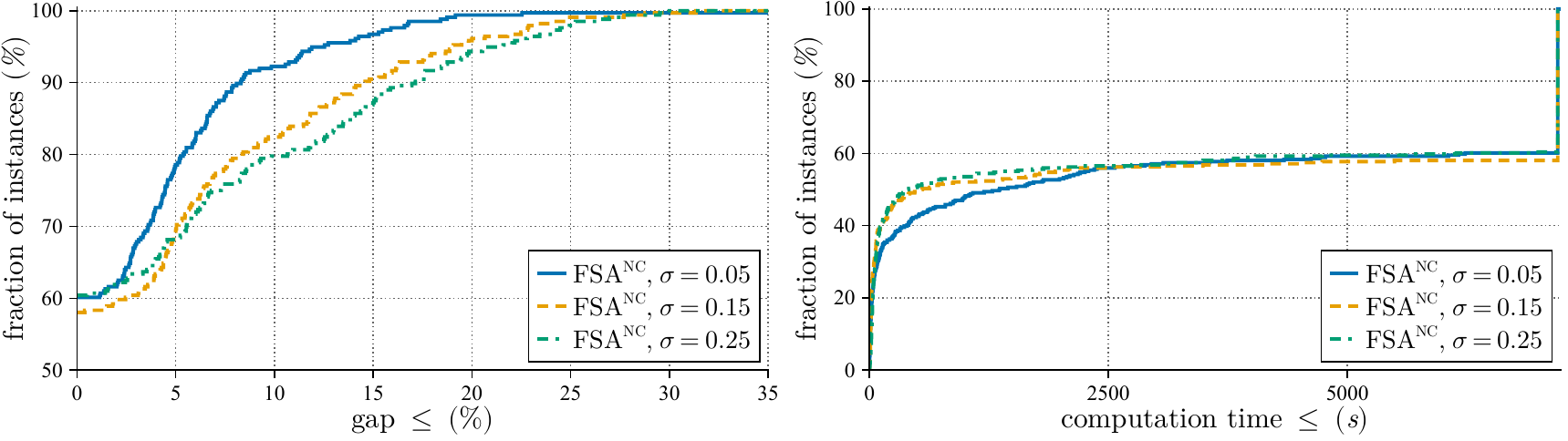}} 
{Optimality gaps and runtimes of \solalgNoBC on instances with 75 tasks for various amounts of temporal dependencies. 
\label{fig:ff:75:gap:time}}
{}
\end{figure}
\begin{table}
\TABLE
{Performance of \solalgNoBC on instances with 75 tasks. 
  \label{tab:meth:FF:key:char:75}}
{\begin{tabular}{ccccc}
\toprule
Temp.\ Dep. &  \#$_\mathrm{opt}$ & time &  gap & gap$_{\mathrm{max}}$    \\    
\midrule 
  syn &   113 & 2614 & 3.53 & 24.07          \\ 
  min &   113 & 2661 & 1.67 & 12.10             \\
  max &  91 & 3533 & 5.05 & 29.84            \\ 
minmax &  111 & 2829 & 1.84 & 17.70             \\ 
  overlap &  101 & 3122 & 4.25 & 29.5           \\ 
  non-overlap &   71 & 4430 & 6.79 & 31.34      \\ 
       \cmidrule(lr){1-1} \cmidrule(lr){2-5}   
    all &     600 & 3198 & 3.86 & 31.34          \\ 
\bottomrule
\end{tabular}}{}
\end{table}

Consistent with the results for smaller instances, a larger number of temporal dependencies results in larger optimality gaps while the impact on runtimes (which slightly decrease) is limited. While the performance of \solalgNoBC naturally decreases compared to smaller instances, it can still solve 600 out of the 1\,008 instances, more than the number of smaller instances solved by any of the two benchmark methods. From \cref{fig:ff:75:gap:time} we also observe that instances that can be solved optimally are typically solved relatively quickly. 
The results in \cref{tab:meth:FF:key:char:75} also confirm the main conclusions drawn in the previous subsection. Again, non-overlapping constraints where only 71 out of 168 instances could be solved are the most challenging ones for \solalgNoBC, followed by maximum precedence and overlap constraints for which it could solve 91 and 101 instances, respectively. The performance of the remaining types (i.e., synchronization, minimum precendence, minmax precedence) with respect to numbers of solved instances and runtimes is very similar. Interestingly, the average and maximum gaps for synchronization are, nevertheless, significantly larger than for the other two types.
For one instance with minmax precedence, \solalgNoBC did not obtain a lower bound within the time limit, and this instance is therefore excluded from the statistics in \cref{tab:meth:FF:key:char:75}.
Overall, these results show that \solalgNoBC can be used to (approximately) solve larger instances with 75 tasks for which either proven optimal results or high-quality solutions and acceptable optimality gaps are returned in the majority of cases. 
Additional details on the impact of the VI and the characteristics of the enumeration phase are presented in \cref{app:ins:FSA}. They show that the impact of the cuts increase with the number of temporal dependencies, and that the TIFI, TDIFI, and RCC have a larger impact for the set of instances with more restrictive time window and capacity restrictions. They also indicate the effectiveness of the fragments reduction procedure.

\section{Conclusion}\label{sec:Conclusion}
This paper studies a generic vehicle routing problem with temporal dependencies (VRPTD) which allows to consider a wide range of highly relevant types of temporal dependencies between pairs of tasks. These temporal dependencies can specify possible starting orders, minimum and maximum permitted differences in starting times, and whether the execution of tasks must (not) overlap. Generalizing exact algorithms for classical VRPs that rely on route-based formulations to the VRPTD while maintaining their efficiency is far from straightforward. We, therefore, propose to represent vehicle routes as sequences of fragments connecting nodes with temporal dependencies. 
We propose a novel fragment-based formulation that combines individual fragments into routes while ensuring feasibility with respect to time windows, capacity constraints, and temporal dependencies. To strengthen the formulation, we propose four families of valid inequalities: fragment-based subtour elimination constraints, time- and temporal dependency infeasible fragment inequalities, and fragment-based (rounded) capacity constraints. Building on this formulation, we propose an exact price-cut-and-enumerate solution method that benefits from the fact that aforementioned route fragments are specifically designed to achieve a computationally beneficial trade-off between the strength of the formulation and the complexity of the pricing problem. The novel solution method first derives a lower bound through alternating column-and-row generation using a specifically designed labeling algorithm and the separation of violated valid inequalities. It then aims to find a feasible solution and iteratively refines both lower and upper bounds through fragment enumeration and branch-and-cut, using an enumeration algorithm and a fragment reduction procedure introduced in this article.

Our computational study shows that our algorithm outperforms two benchmark methods that are also applicable to a wide range of temporal dependencies. Here, we consider a branch-and-price algorithm by \cite{Dohn2011}, which enforces temporal dependencies through time window branching but cannot deal with non-overlapping constraints and an approach based on solving an arc-based formulation using Gurobi. Our approach solves significantly more benchmark instances in less time and derives smaller optimality gaps for instances that cannot be solved within the given time limit. It can also tackle instances whose size is prohibitive for the two benchmark methods. Furthermore, our results 
confirm the positive impact of the proposed valid inequalities on the performance of our algorithm. 

This paper demonstrates the successful application of fragment-based methods to a highly relevant class of routing problem conceptually different from formulations in other studies.
Our computational results do also indicate several suggestions for further improvements. These include integration of (advanced) primal heuristics, enhanced fragment reduction techniques, or the use of a fragment pool \citep{Yang2023} to reduce the number of enumeration procedures. 
Furthermore, parallelization of the labeling and enumeration algorithm could reduce the computation times. Interesting avenues for further research also include the identification of additional valid inequalities and to derive separation routines for fragment-based subtour elimination constraints and fragment-based (rounded) capacity constraint. 
Finally, studying the use of a fragment-based formulation in settings with uncertain travel and service times or dynamic requests as well as the use of machine learning techniques to identify promising fragments or 
reduce the number of enumerated fragments \citep{SIPPEL2025107123} appear to be promising directions for future research. 


\ACKNOWLEDGMENT{
The publication used the Dutch National Supercomputer Snellius that is (co-)funded by the Dutch Research Council (NWO), hosted by SURF through the call for proposals 'Computing Time on National Computer Facilities'. 
%
It used the Dutch national e-infrastructure with the support of the 
SURF Cooperative using grant no. EINF-15088.
}


%
%
%

\bibliographystyle{informs2014} 
\bibliography{references} 

\begin{APPENDICES}

\section{Notation}\label[appsec]{app:notation}
\cref{tab:notation:gen} provides a summary of the sets and parameters used within the problem definition and to introduce the fragment-based formulation~\eqref{eq:ff}. \cref{tab:notation:alg} summarizes the notation used in the \solalg. 

\begin{table}[H]
\TABLE
{Notation used for the \probabbr and fragment-based formulation~\eqref{eq:ff}. 
  \label{tab:notation:gen}}
{\begin{tabular}{ll}
\toprule
Sets & Description \\ 
\midrule 
	$\Veh$             & Vehicles  \\ 
    $V$                & Tasks  \\ 
    $N$                & Tasks and the depot ($V \cup \{0\}$) \\
    $D$                & Task pairs with temporal dependencies ($D \subseteq V \times V$)  \\
    $\TDVis$            & Tasks with temporal dependencies ($\TDVis \subseteq V$)\\   
    $\AllFrag$         & Fragments   \\ 
    $\FragFromTo{u}{v}$ & Fragments starting at $u \in \TDVis$ and ending at $v \in \TDVis$ ($\FragFromTo{u}{v} \subseteq \AllFrag$)  \\ 
    $\FragFrom{v}$         & Fragments starting at $v \in \TDVis$ ($\FragFrom{v} \subseteq \AllFrag$)   \\ 
    $\FragTo{v}$        & Fragments ending at $v \in \TDVis$ ($\FragTo{v} \subseteq \AllFrag$)   \\
    \midrule 
Parameters & Description \\  
\midrule 
$\VehCap$            & Vehicle capacity  \\ 
$[0,\Tmax]$           & Planning horizon \\
$[\alpha_v,\beta_v]$ & Time window during which task $v \in N$ may start  \\ 
$d_v$  & Duration of task $v \in V$\\ 
$q_v$  & Demand of task $v \in V$ \\ 
$t_{uv}$  & Travel time from task $u \in V$ to task $v \in V$\\
$c_{uv}$  & Travel cost from task $u \in V$ to task $v \in V$ \\ 
$\dmin{uv}, \dmax{uv}$  & Min / max allowed difference between the start of task $u$ and $v$, 
$\{u,v\} \in D$, if $u$ start no later than $v$ \\ 
$\lsF{v}{F}$  & Latest starting time of fragment $F \in \AllFrag$   \\ 
$\esF{v}{F}$  &  Earliest completion time of fragment $F \in \AllFrag$\\ 
$\durF{F}$  & Minimum duration of fragment $F \in \AllFrag$ \\ 
$\demF{F}$  &  Total demand of fragment $F \in \AllFrag$ \\ 
$\costF{F}$  & Total travel cost of fragment $F \in \AllFrag$ \\ 
\bottomrule
\end{tabular}}{}
\end{table} 

\begin{table}[H]
\TABLE
{Notation used for the \solalg. 
  \label{tab:notation:alg}}
{\begin{tabular}{ll}
\toprule
Sets & Description  \\ 
\midrule 
$\CGfrag$             & Fragments obtained during column-and-row generation in Step~\ref{sol:proc:step:lb} \\ 
$\CGfragFromTo{u}{v}$  & Fragments in $\CGfrag$ starting at $u \in \TDVis$ and ending at $v \in \TDVis$ \\ 
$\ENfrag$   & Fragments obtained during the enumeration procedure in Step~\ref{sol:proc:step:enum}  \\ 
$\ENfragTo{v}$   & Fragments in $\ENfrag$ ending at $v \in \TDVis$ \\ 
$\ENfragFrom{v}$   &  Fragments in $\ENfrag$ starting at $v \in \TDVis$ \\ 
$\IncSol$ & Fragments best feasible solution found \\
\midrule 
Parameters & Description  \\ 
\midrule 
$\LBalg$ & Lower bound for solution value\\
$\IncSolVal$ & Upper bound for solution value \\ 
$\UBguess$ & Candidate upper bound for solution value \\ 
$\UBinitstep$ & Maximum gap for initial candidate upper bound  \\
$\UBstep$ & Maximal increase of candidate upper bound \\ 
$k_{\mathrm{max}}$ & Maximum cardinality considered when enumerating sets to separate FSECs  \\
$\CGtimeMIP$ & Time limit for finding an initial upper bound in Step~\ref{sol:proc:step:init:ub} \\ 
$\maxfragenum$ & Maximum number of identified fragments before \solalg stops   \\ 
\midrule 
Other & Description  \\ 
\midrule 
\FF     &    Fragment-based formulation~\eqref{eq:ff} with constraints \eqref{eq:ff:time:mtz}-\eqref{eq:ff:time:b}, 
\eqref{eq:ff:td:uv:le}-\eqref{eq:ff:td:p}, and \eqref{eq:ff:load:mtz}-\eqref{eq:ff:load:l} \\ 
\FFVI    &   \FF strengthened by inequalities \eqref{eq:FSEC}, \eqref{eq:InfPathV}, \eqref{eq:InfPathTD}, and \eqref{eq:RCC} \\ 
\FFVIm   &   \FF strengthened by constraints \eqref{eq:FSEC}, \eqref{eq:InfPathV}, \eqref{eq:InfPathTD}, and \eqref{eq:FRCC}  \\ 
$\REL{\mathcal{X}}$ & Linear relaxation formulation $\mathcal{X}$ \\ 
$\FORM{\mathcal{X}}[\AllFrag']$ & Formulation $\mathcal{X}$ restricted to the set of fragments $\AllFrag'\subseteq \AllFrag$ \\ 
$\VAL[\mathcal{X}]$ & Solution value formulation $\mathcal{X}$ ($+\infty$ if no feasible solution is found)    \\ 
\bottomrule
\end{tabular}}{}
\end{table}

\section{Proofs of statements}\label[appsec]{app:proofs}

This section provides the proofs of all theorems, propositions, and corollaries presented in the paper. For convenience, each result is restated before its proof. 

\fragderiveschedcond*

\begin{proof}{Proof.}
Consider a fragment $F=(1,\dots,v_{\lenF{F}})$. We first show that if there are no temporal dependencies between $v_{1}$ and $v_{\lenF{F}}$, i.e., $\{v_1,v_{\lenF{F}}\} \notin D$, then for each $i=1,\dots,\lenF{F}$:   
    \begin{enumerate}[(i)]
        \item $\esF{v_{i}}{F[i]}$ is the earliest \emph{starting time} of $v_i$ after performing $F[i]$, \label{pf:frag:char:cond:es} 
        \item $\lsF{v_1}{F[i]}$ is the latest \emph{starting time} at $v_1$ when performing $F[i]$, \label{pf:frag:char:cond:ls} and 
        \item $\durF{F[i]}$ is the minimum duration of $F[i]$ and can be achieved for each starting time of $v_1$ in $[\esF{v_{i}}{F[i]} - \durF{F[i]},\lsF{v_1}{F[i]}]$.\label{pf:frag:char:cond:dur}      
    \end{enumerate}   
The initialization $\esF{v_{1}}{F[1]}=\alpha_{v_1}$, $\lsF{v_1}{F[1]}=\beta_{v_1}$, and $\durF{F[1]}=0$ satisfies conditions \eqref{pf:frag:char:cond:es}-\eqref{pf:frag:char:cond:dur}. Assuming these conditions hold for $F[i-1]$ with $i \in \{2,\dots \lenF{F}\}$, it follows by induction that: 
\begin{itemize}
    \item $\esF{v_{i}}{F[i]} = \max\{\esF{v_{i-1}}{F[i-1]} + d_{v_{i-1}} + t_{v_{i-1} v_{i}}, \alpha_{v_{i}} \}$ is the earliest completion time of $F[i]$, since $F[i-1]$ cannot be completed before $\esF{v_{i-1}}{F[i-1]}$; 
    \item $\lsF{v_{i}}{F[i]}= \min\{\lsF{v_1}{F[i-1]}, \beta_{v_i} - t_{v_{i-1} v_{i}} - d_{v_{i-1}} - \durF{F[i-1]}\}$ is the latest feasible starting time at $v_1$ for $F[i]$, since $v_{i-1}$ must start no later than $\beta_{v_i} - t_{v_{i-1} v_{i}} - d_{v_{i-1}}$ to reach $v_i$ within its time window; and
    \item $\durF{F[i]} = \max\{ \durF{F[i-1]} + d_{v_{i-1}} + t_{v_{i-1} v_{i}}, \alpha_{v_{i}} - \lsF{v_1}{F[i-1]}\}$ is the minimum duration of $F[i]$, since starting $v_1$ as late as possible, i.e., at $\lsF{v_{i}}{F[i]}$, implies $v_{i-1}$ cannot start before $\lsF{v_1}{F[i]} + \durF{F[i-1]}$. 
    If $\durF{F[i]} = \alpha_{v_{i}} - \lsF{v_1}{F[i-1]}$, then $\esF{v_{i}}{F[i]} = \alpha_{v_i}$, and this duration is achieved only for starting time $\lsF{v_1}{F[i]}$. Otherwise, the minimum duration $\durF{F[i]} =\durF{F[i-1]} + d_{v_{i-1}} + t_{v_{i-1} v_{i}}$ can be achieved for each starting time of $F[i]$ within the interval $[\esF{v_{i}}{F[i]} - \durF{F[i]},\lsF{v_1}{F[i]}]$.
\end{itemize}

    Temporal dependencies between $v_{1}$ and $v_{\lenF{F}}$ additionally impose that the minimum duration of $F$ must be in  
    $[\dmin{v_1 v_{\lenF{F}}}, \dmax{v_1 v_{\lenF{F}}}]$. If the standard updating rules yield $\durF{F[\lenF{F}]} \in [\dmin{v_1 v_{\lenF{F}}}, \dmax{v_1 v_{\lenF{F}}}]$, the temporal dependencies are not restrictive and the parameters are valid, whereas $\durF{F[\lenF{F}]} > \dmax{v_1 v_{\lenF{F}}}$ implies that the fragment is infeasible.  
    If $\durF{F[\lenF{F}]} < \dmin{v_1 v_{\lenF{F}}}$, the minimum duration must be increased to $\dmin{v_1 v_{\lenF{F}}}$. This higher minimum duration implies that $F$ cannot start later than $\beta_{v_{\lenF{F}}} - \dmin{v_1 v_{\lenF{F}}}$ and cannot finish before $\alpha_{v_1} + \dmin{v_1 v_{\lenF{F}}}$, and can only be achieved for each starting time within $[\max\{\alpha_{v_1},\esF{}{F[\lenF{F}]} - \dmin{v_1 v_{\lenF{F}}}\}, \min \{\lsF{}{F[\lenF{F}]} , \beta_{v_{\lenF{F}}} - \dmin{v_1 v_{\lenF{F}}}\}]$.

\end{proof}

\fragdur*
\begin{proof}{Proof.}
Consider a fragment $F=(v_1,\dots,v_{\lenF{F}}) \in \AllFrag$. From  \cref{prop:frag:derive:sched:cond} and its proof, we observe that a schedule with duration $\durF{F}$ can be achieved for each starting time in $[\esF{}{F} - \durF{F}, \lsF{}{F}]$. For each starting time $t \in [\max\{\alpha_{v_1}, \esF{}{F} - \dmin{v_1 v_{\lenF{F}}}\}, \esF{}{F} - \durF{F})$, the additional mandatory waiting time resulting from starting a fragment at a time earlier than the one 
resulting in a minimum duration is equal to $(\esF{}{F} - \durF{F}) - t$. The duration of a compact schedule for a starting time $t \in [\max\{\alpha_{v_1}, \esF{}{F} - \dmin{v_1 v_{\lenF{F}}}\}, \esF{}{F} - \durF{F})$ is, therefore, equal to $\esF{}{F} - t$.
\end{proof}

\numtifi*
\begin{proof}{Proof.}
Recall that the left-hand side of each TIFI for task $v\in \TDVis$ and a given time $t\in [\alpha_v,\beta_v]$ contains the variables associated to all fragments ending at $v$ whose earliest completion time is larger than or equal to $t$ and to all fragments starting at $v$ whose latest starting time is smaller than $t$, i.e., $\{F \in \FragTo{v} : \esF{v}{F} \geq t\}$ and $\{F \in \FragFrom{v} : \lsF{v}{F} < t\}$. Obviously, the latter set is non-decreasing with $t$. Now, consider two consecutive earliest completion times $t^1, t^2\in  \{\esF{v}{F} : F \in \FragTo{v}\}$, i.e., $t^1<t^2$ such that there does not exist an earliest completion time $t^3\in \{\esF{v}{F} : F \in \FragTo{v}\}$ with $t^1<t^3 <t^2$. Then, the sets $\{F \in \FragTo{v} : \esF{v}{F} \geq t\}$ are identical for each $t\in (t^1,t^2)$ and a subset of $\{F \in \FragTo{v} : \esF{v}{F} \geq t^2\}$. Together with the aforementioned monotonicity of $\{F \in \FragFrom{v} : \lsF{v}{F} < t\}$, this implies that each TIFI for node $v\in \TDVis$ whose time $t\in [\alpha_v,\beta_v]$ is smaller than the latest completion time (i.e., $t<\max \{\esF{v}{F} : F \in \FragTo{v}\}$) and does not correspond to a completion time (i.e., $t\notin \{\esF{v}{F} : F\in \FragTo{v}\}$) is implied by the TIFI associated to the smallest completion time larger than $t$. TIFIs for times $t\in (max \{\esF{v}{F} : F \in \FragTo{v}\},\beta_v]$ are implied from constraints~\eqref{eq:ff:ass} and \eqref{eq:ff:flow} which ensure $\sum_{F \in \FragFrom{v} } x_F\le 1$. \Halmos
\end{proof}

\numtdifi*
\begin{proof}{Proof.}
Consider a pair $\{u,v\}\in D$ and observe that each of the four considered types of inequalities considers the variables associated to two sets of fragments on its left-hand side. As we observe below, in each of the four cases one of these sets has a certain monotonicity property while the other remains constant between two consecutive time points of the relevant set stated in the theorem. 
\begin{itemize}
	\item For inequalities~\eqref{eq:InfPathTD:uv:min}, the set $\{F \in \FragFrom{v} : \lsF{v}{F} < t+\dmin{uv}\}$ is non-decreasing in time $t$, while the set
	$\{F \in \FragTo{u} : \esF{u}{F}\ge t\}$ is non-increasing in time $t$ and remains constant between two consecutive earliest completion times at $u$. Thus, each inequality~\eqref{eq:InfPathTD:uv:min} for a time $t$ between two consecutive earliest completion times is dominated by the inequality for the latest of those two points.
	\item For inequalities~\eqref{eq:InfPathTD:uv:max}, the set $\{F \in \FragTo{v}: \esF{v}{F} > t+\dmax{uv}\}$ is non-increasing in time $t$, while $\{F \in \FragFrom{u}: \lsF{u}{F} \leq t\}$ is non-decreasing in time $t$ and remains constant between two consecutive latest starting times at $u$. Any inequality~\eqref{eq:InfPathTD:uv:max} for a time between two consecutive latest starting times is, therefore, dominated by the inequality for the earliest of these two points. 	
	\item For inequalities~\eqref{eq:InfPathTD:vu:min}, the set $\{F \in \FragFrom{u} : \lsF{u}{F} < t+\dmin{vu}\}$ is non-decreasing in time $t$ while the set $\{F \in \FragTo{v} : \esF{v}{F}\ge t\}$ is non-increasing in time $t$ and remains constant between two consecutive earliest completion times at $v$. Hence, each inequality~\eqref{eq:InfPathTD:vu:min} for a time $t$ between two consecutive earliest completion times is dominated by the latter of the two. 
	\item For inequalities~\eqref{eq:InfPathTD:vu:max}, the set $\{F \in \FragTo{u}: \esF{u}{F} > t+\dmax{vu}\}$ is non-increasing in time $t$, while $\{F \in \FragFrom{v}: \lsF{v}{F} \leq t\}$ is non-decreasing in time $t$ and remains constant between two consecutive latest starting times at $v$. Any inequality~\eqref{eq:InfPathTD:vu:max} for a time between two consecutive latest starting times is, thus, dominated by the inequality for the earliest of the two.
\end{itemize}

In addition, observe that inequalities \eqref{eq:InfPathTD:uv:max} and \eqref{eq:InfPathTD:vu:max} for time points before the earliest in their relevant set, and inequalities \eqref{eq:InfPathTD:uv:min} and \eqref{eq:InfPathTD:vu:min} for time points after the latest in their relevant set, are implied since $\sum_{F \in \FragFrom{w} } x_F\le 1$ and $\sum_{F \in \FragFrom{w} } x_F\le 1$ holds for $w \in \{u,v\}$ due to constraints~\eqref{eq:ff:ass} and \eqref{eq:ff:flow}.  
Hence, the validity of the theorem for all four sets of inequalities follows using the same argumentation as in the proof of \cref{thm:InfPathV:num}.
\end{proof}

\dpdominanceboundcomplcost*

\begin{proof}{Proof.}
Consider two forward paths $f$ and $g$ representing incomplete fragments such that their corresponding labels satisfy all the dominance criteria in favor of $f$ except for the reduced cost, i.e., $\flCost{f} > \flCost{g}$.
Thus, $\flLS{f\oplus \hat{f}} \geq \flLS{g\oplus \hat{f}}$, $\flDur{f\oplus \hat{f}} \leq \flDur{g\oplus \hat{f}}$, and $\flLoad{f\oplus \hat{f}} \leq \flLoad{g\oplus \hat{f}}$ hold for any feasible extension $\hat{f}$ of $g$ such that $g \oplus \hat{f}$ is a fragment.  

We show that $\rcBoundDom{f}{g}$ as defined in \eqref{eq:def:lb:compl:cost} satisfies condition~\eqref{eq:cond:lb:compl:cost} by separately deriving lower bounds for the terms associated with $\dTimeUB{\flStart{f}}$ and $\dLoadUB{\flStart{f}}$, i.e., for $\flLS{f \oplus \hat{f}} - \flLS{g \oplus \hat{f}}$ and $\flLoad{g \oplus \hat{f}} - \flLoad{f \oplus \hat{f}}$.

\medskip  

\noindent\textit{Term associated with $\dTimeUB{\flStart{f}}$.}
We show that $\flLS{f \oplus \hat{f}} - \flLS{g \oplus \hat{f}} \geq \min\{ \flLS{f} - \flLS{g},\; (\flLS{g} + \flDur{g}) - \flDur{f} - \flLS{g},\; \max \{ \flLS{g},\flLSDMin{g} \} - \flLS{g} \}$ for any feasible extension $\hat{f}= (v_1, \dots, v_{\ell})$ of $g$, by considering three mutually exclusive cases:
\begin{enumerate}[(i)]
    \item\label{pf:comp:bound:ls:same} $\flLS{f \oplus \hat{f}} = \flLS{f}$: Since the latest starting time of $g$ can only decrease after an extension, i.e., $\flLS{g \oplus \hat{f}} \leq \flLS{g}$, we obtain $\flLS{f \oplus \hat{f}} - \flLS{g \oplus \hat{f}} \geq \flLS{f} - \flLS{g}$.
    
    \item\label{pf:comp:bound:ls:tw} $\flLS{f \oplus \hat{f}} < \flLS{f}$ and $\flLS{f \oplus \hat{f}}$ is determined by the time windows of the tasks within $\hat{f}$: The time window of any task $v_i \in \hat{f}$ imposes a latest starting time at $\flStart{f}$ that depends on the duration of the partial path up to $v_i$. Let $v_{i} \in \hat{f}$ be the task determining the latest starting time of $f \oplus \hat{f}$, i.e., $\flLS{f \oplus \hat{f}} = \flLS{f \oplus (v_1,\dots,v_i)} = \beta_{v_i} - \flDur{f \oplus (v_1,\dots,v_i)}$. Since extending $f$ with $(v_1,\dots,v_i)$ reduces the latest starting time at $\flStart{f}$, it holds that $\flDur{f \oplus (v_1,\dots,v_i)} = \flDur{f} + \sum_{j=0}^{i-1} (d_{v_j} + t_{v_j v_{j+1}})$ with $v_{0}= \flEnd{f}$. Since $\flDur{g \oplus (v_1,\dots,v_i)} \geq  \flDur{g} + \sum_{j=0}^{i-1} (d_{v_j} + t_{v_j v_{j+1}})$, it follows that $\flDur{g \oplus (v_1,\dots,v_i)} - \flDur{f \oplus (v_1,\dots,v_i)} \geq \flDur{g} - \flDur{f}$. Combined with $\flLS{g \oplus \hat{f}} \leq \flLS{g}$, we obtain $\flLS{f \oplus \hat{f}} - \flLS{g \oplus \hat{f}} \geq \flLS{f \oplus (v_1,\dots,v_i)} - \flLS{g \oplus (v_1,\dots,v_i)} \geq  \flDur{g} - \flDur{f} = (\flLS{g} + \flDur{g}) - \flDur{f} - \flLS{g}$. 

    \item\label{pf:comp:bound:ls:dmin} $\flLS{f \oplus \hat{f}} < \flLS{f}$ and $\flLS{f \oplus \hat{f}}$ is determined by a minimum time difference $\dmin{\flStart{f} v_{\ell}}$: Extension $\hat{f}$ imposes $\flLS{g \oplus \hat{f}} \leq \beta_{v_{\ell}} - \dmin{\flStart{g} v_{\ell}}$. 
    Combined with $\flLS{g \oplus \hat{f}} \leq \flLS{g}$, we get $\flLS{g \oplus \hat{f}} \leq \min\{\flLS{g}, \beta_{v_{\ell}} - \dmin{\flStart{g} v_{\ell}}\}$. 
    Since the same minimum duration constraint implies $\flLS{f \oplus \hat{f}} = \beta_{v_{\ell}} - \dmin{\flStart{f} v_{\ell}} = \beta_{v_{\ell}} - \dmin{\flStart{g} v_{\ell}}$, we obtain $\flLS{f \oplus \hat{f}} - \flLS{g \oplus \hat{f}} \geq \beta_{v_{\ell}} - \dmin{\flStart{g} v_{\ell}} - \min\{\flLS{g}, \beta_{v_{\ell}} - \dmin{\flStart{g} v_{\ell}}\} = \max \{ \flLS{g}, \beta_{v_{\ell}} - \dmin{\flStart{g} v_{\ell}} \} - \flLS{g} \geq \max \{ \flLS{g}, \min_{u \in V_{\dmin{}}^{g}} \{\beta_{u} - \dmin{\flStart{g} u} \} \} - \flLS{g} = \max \{ \flLS{g}, \flLSDMin{g}\} - \flLS{g}.$
\end{enumerate}
Taking the minimum over the three cases and multiplying by $\dTimeUB{\flStart{f}} \geq 0$ yields the first term in \eqref{eq:def:lb:compl:cost}.

\medskip 

\noindent\textit{Term associated with $\dLoadUB{\flStart{f}}$.}
Any feasible extension $\hat{f}$ increases the collected demand of both $f$ and $g$ by the same amount. Therefore, $\flLoad{g \oplus \hat{f}} - \flLoad{f \oplus \hat{f}} = \flLoad{g} - \flLoad{f}$ for all $\hat{f}$, yielding the second term $(\flLoad{g} - \flLoad{f})\dLoadUB{\flStart{f}}$. \Halmos
\end{proof}

\section{Pre-processing}\label[appsec]{app:pre:procc}
The pre-processing applied at the start of the \solalg
\begin{inparaenum}[(i)]
    \item updates the time windows of tasks by accounting for the travel time from and to the depot,
    \item extends the set of temporal dependencies, and
    \item tightens the time windows of tasks with temporal dependencies based on those restrictions.
\end{inparaenum}

First, the time window of each task $v \in V$ is updated to $[\max\{\alpha_u, t_{0v}\},\min\{\beta_v,\Tmax - t_{v0} -d_{v}\}]$, as $t_{0v}$ and $\Tmax - t_{v0} -d_{v}$ are the earliest and latest starting times of $v$ that fit within the planning horizon $[0,\Tmax]$. 

Second, the pre-processing extends the set $D$ of temporal dependencies. The set $D$ induces a temporal dependency graph $(\TDVis, D)$. 
 The connected components of this graph indicate groups of tasks that are linked through temporal dependencies. A temporal dependency is added between each pair of tasks within such a group.  
Specifically, consider tasks $u,v,w \in \TDVis$ with $\{u,v\}, \{v,w\} \in D$. \cref{tab:preprocc:td} describes how each combination of the starting order of $\{u,v\}$ and $\{v,w\}$ implies temporal restrictions for each starting order of $u$ and $w$. Note that not every combination is feasible, for instance, if the implied maximum difference is negative. We obtain a temporal dependency between $\{u,w\}$ defined by $\dminAll$ and $\dmaxAll$ for each starting order of $u$ and $w$, using the implied lower and upper bound on the minimum and maximum difference in starting time from the set of feasible combinations. This yields new temporal dependencies and potentially strengthens existing ones. 

After extending the set of temporal dependencies $D$, we perform time window tightening following \cite{VANMONTFORT2025104235} and inspired by \cite{RASMUSSEN2012598, Dohn2011}. For each pair $\{u,v\} \in D$ without a pre-defined order in which the two tasks must be performed, we update the time window of $u$ to $[\max\{\alpha_u, \alpha_v - \dmax{uv}\},\min\{\beta_u, \beta_v + \dmax{vu}\}]$ and that of $v$ to $[\max\{\alpha_v, \alpha_u - \dmax{vu}\},\min\{\beta_v, \beta_u + \dmax{uv}\}]$. 
If $u$ must start before $v$, the time windows are updated to $[\max\{\alpha_u, \alpha_v - \dmax{uv}
\}, \min\{\beta_u, \beta_v - \dmin{uv}\}]$ for $u$ and $[\max\{\alpha_v, \alpha_u + \dmin{uv}\},\min\{\beta_v, \beta_u + \dmax{uv}\}]$ for $v$.

\begin{table}
\TABLE
{The minimum (min diff) and maximum (max diff) implied differences in starting times of $u$ and $w$ for each starting order, based on the starting order of tasks $\{u,v\} \in D$ and $\{v,w\} \in D$. \label{tab:preprocc:td}}
{\begin{tabular}{ccccc}
\toprule
& \multicolumn{2}{c}{ $u \preceq w$} & \multicolumn{2}{c}{$w \preceq u$} \\
 \cmidrule(lr){2-3} \cmidrule(lr){4-5}	
 & min diff & max diff & min diff & max diff \\
\midrule
$u \preceq v \preceq w$ & $\dmin{uv} + \dmin{vw}$ & $\dmax{uv} + \dmax{vw}$ & $\Tmax$ & $\Tmax$ \\ 
$u \preceq v, w \preceq v$ & $\dmin{uv} - \dmax{wv}$ & $\dmax{uv} - \dmin{wv}$ & $\dmin{wv} - \dmax{uv} $ & $\dmax{wv} - \dmin{uv}$ \\
$v \preceq u, v \preceq w$ & $-\dmax{vu} + \dmin{vw}$ & $-\dmin{vu} +\dmax{vw}$ & $-\dmax{vw} + \dmin{vu}$ & $-\dmin{vw} + \dmax{vu}$ \\
$w \preceq v \preceq u$ & $\Tmax$ & $\Tmax$ & $\dmin{wv} + \dmin{vu}$ & $\dmax{wv} + \dmax{vu}$ \\ 
\bottomrule
\end{tabular}}{}
\end{table}

\section{Impact of valid inequalities on the labeling algorithm}\label[appsec]{app:lab:alg:vi}
This section discusses the impact of the row generation on the labeling algorithm. As described in \cref{sss:solappr:lb:vi}, \solalg separates and adds violated FSECs~\eqref{eq:FSEC}, TIFIs~\eqref{eq:InfPathV}, TDIFIs~\eqref{eq:InfPathTD}, and RCCs~\eqref{eq:RCC}. 
We describe how these constraints affect the reduced cost of a fragment and how the associated dual costs are integrated within the labeling algorithm.

\paragraph{FSECs and RCCs}
The FSECs and RCCs are robust \citep{Pessoa2020}. Therefore, we simply subtract the dual variable corresponding to each included FSEC defined for a set $S \subseteq \TDVis$ from the completion costs of fragments starting and ending in $S$, i.e. fragments $F \in \FragFromTo{S}{S}$. We also subtract the dual variables corresponding to each included RCC for a set $S \subseteq V$ from the reduced cost of all arcs entering $S$. 

\paragraph{TIFIs and TDIFIs}
Dual variables corresponding to TIFIs and TDIFIs are included in the completion costs of fragments as follows. Consider a TIFI defined for a task $v \in \TDVis$ and time $t \in [\alpha_v, \beta_{v}]$, and a fragment $F \in \FragTo{v} \cup \FragFrom{v}$. The dual variable of this TIFI is subtracted from $\rcComplFrag{F}$ if: 
\begin{itemize}
    \item $F$ has an earliest starting time at $v$ of at least $t$ ($\esF{}{F} \geq t$) , or
    \item $F$ has a latest starting time at $v$ of at most $t$ ($\lsF{}{F} < t$).
\end{itemize}
For TDIFIs, a similar argument applies. Consider a pair of tasks $\{u,v\} \in D$ and a TDIFI of type \eqref{eq:InfPathTD:uv:min} or \eqref{eq:InfPathTD:uv:max} defined for time $t \in [\alpha_u,\beta_u]$, or a TDIFI of type \eqref{eq:InfPathTD:vu:min} or \eqref{eq:InfPathTD:vu:max} for time $t \in [\alpha_v, \beta_v]$. Given a fragment $F \in \FragTo{u} \cup \FragFrom{u} \cup \FragTo{v} \cup \FragFrom{v}$, the dual variable of this TDIFI is subtracted from $\rcComplFrag{F}$ if its earliest starting time or latest completion time at $u$ and / or $v$ satisfy the conditions defined by the TDIFI at time $t$. 
\paragraph{Bound $\rcBoundDom{f}{g}$} The bound on the difference in completion costs between two fragments $f$ and $g$ as defined in \cref{thm:dp:dominance:bound:compl:cost} is also affected by the TIFIs and TDIFIs. Therefore, this bound is strengthened by including the effects of the dual variables of these cuts.

\section{Benchmark approaches}\label[appsec]{app:bench:appr}
This section describes the two benchmark methods to which we compare the performance of the \solalg. \cref{app:time:window:branch} details the re-implementation of the branch-and-price (BP) approach of \cite{Dohn2011} based on time window branching, which can handle a broad range of temporal dependency types but is not applicable to non-overlapping dependencies. \cref{app:af} presents the arc-based formulation that can handle all types of temporal dependencies considered in this article and is solved with a general purpose solver. 

\subsection{Branch-and-price based on time window branching}\label[appsec]{app:time:window:branch}

We re-implemented the BP approach from \citet{Dohn2011} as the original source code is not available. We acknowledge that such a re-implementation may (slightly) differ from the original method in particular w.r.t.\ aspects that could not fully described in the original paper (due to space limitation). We do also believe, however, that such a re-implementation also supports a fair comparison using the  same hardware, programming language, and the same approach in implementing the labeling algorithm for solving the pricing subproblems. Our re-implementation of the BP approach based on time windows branching follows the original work of \citet{Dohn2011} as closely as possible and we, therefore, refrain from giving a full description of their method. Instead, we detail our implementation of aspects that were not fully described in \citet{Dohn2011} and improvements made by us to ensure a fair comparison. Since they do not specify their labeling algorithm, we had to implement the labeling algorithm used for solving the pricing subproblem from scratch. To facilitate a fair comparison, we used a straightforward adaptation of the one used for our own algorithm to the case of full routes without inter-route or intra-route temporal dependency restrictions.   
Thus, as opposed to the original paper by \citet{Dohn2011}, our re-implementation also uses the $ng$-path relaxation which is, by now, a standard technique known to improve the performance of route-based BP approaches. Experiments were also performed for an adjusted version of the labeling algorithm mentioned by \citet{Dohn2011}, which excludes routes that contain pairs of tasks that are mutually exclusive due to their temporal dependency restrictions. The latter has, however, not been used in the results reported in this article as preliminary experiments showed that this did not uniformly improve the results due to the slower labeling algorithm. 
These experiments were also used to identify promising values of other parameters for our computational study. In particular, we set the size of the $ng$-neighborhood to $9$ and the maximum number of columns added in each pricing iteration to $50$.

Our implementation also uses the pre-processing procedure presented in \cite{Dohn2011} that aims to extend and strengthen the set of temporal restrictions. 
The branching rules also follow the description of \cite{Dohn2011}, whereby time window branching is performed on tasks with temporal dependencies at which a violation of a temporal dependency is observed. 
Using a strategy similar to our own algorithm, the restricted master problem is initialized using all single task routes and an artificial route with very high costs that satisfies all constraints and therefore ensures feasibility.

\subsection{Arc-based formulation}\label[appsec]{app:af}
 
The arc-based formulation~\eqref{eq:af} can handle all types of temporal dependencies considered in this article. We solve it using a general purpose solver after performing the pre-processing steps specified in \cref{app:pre:procc}.

Formulation~\eqref{eq:af} is defined on graph $G= (N,A)$ where $A$ represents all feasible travels between pairs of tasks. Specifically, $A = \{(u,v) \in N \times N : u \neq v, \; \alpha_{u} +  d_u + t_{uv} \leq \beta_v, \; q_{u} + q_{v} \leq \VehCap\} \setminus \{(u,v) \in V \times V : \{u,v\} \in D, \; \beta_v - \alpha_{u} < \dmin{uv}, \; \max\{d_u + t_{uv}, \alpha_v - \beta_u\} > \dmax{uv} \}$ contains all travels that are feasible with respect to the time windows, capacity constraints, and temporal dependencies. Formulation~\eqref{eq:af} uses a binary arc variable $x_{uv} \in \{0,1\}$ to indicate whether arc $(u,v)\in A$ is selected. For each task $v \in V$, variable $b_v\ge 0$ represents the starting time, while variable $l_v$ indicates the demand served by the vehicle after performing $v$. Binary variables $p_{uv} \in \{0,1\}$ indicate the order of tasks $\{u,v\}\in D$, where $p_{uv} = 1$ indicates that $u$ starts before or at the same time as $v$. 
\begin{subequations}\label{eq:af}
\begin{align}
\min \quad   & \sum_{a \in A} c_{a} x_{a} \label{eq:af:obj}  \\
\mbox{s.t.}\quad & \sum_{(0,v)\in A} 
     x_{0v} \leq |\Veh| 
     \label{eq:af:vehicle:limit} \\ 
& \sum_{(u,v)\in A} x_{uv}  = 1 & \forall v \in V   
\label{eq:af:ass} \\
& \sum_{(u,v)\in A} x_{uv}  - \sum_{(v,u)\in A} x_{vu} = 0 & \forall v \in V \label{eq:af:flow} \\ 
&  b_{u} + (d_{u} + t_{uv}) x_{uv} \leq   b_{v} + (\beta_{u}-\alpha_v) (1- x_{uv}) & \forall (u,v) \in A : u,v \in V
\label{eq:af:time:mtz} \\  
& b_v - b_u \le \delta^{\mathrm{max}}_{uv} p_{uv}           & \forall \{u,v\}\in D, u<v \label{eq:af:td:uv:le}\\  
& b_u - b_v \le \delta^{\mathrm{max}}_{vu}  (1-p_{uv})               & \forall \{u,v\}\in D, u<v \label{eq:af:td:vu:le} \\             
& b_v - b_u \ge \delta^{\mathrm{min}}_{uv} p_{uv} - \max\{0,\beta_u-\alpha_v\} (1 - p_{uv})   &          \forall \{u,v\}\in D, u<v \label{eq:af:td:uv:ge}\\     
& b_u - b_v \ge \delta^{\mathrm{min}}_{vu} (1 - p_{uv})  -\max\{0,\beta_v-\alpha_u\} p_{uv} & \forall \{u,v\}\in D, u<v \label{eq:af:td:vu:ge}  \\ 
& l_{u} + q_v x_{uv} \leq l_v + (\VehCap - q_v) (1-x_{uv}) & \forall (u,v) \in A:  u,v \in V   \label{eq:af:load:mtz}  \\ 
& x_{uv} \in  \{0,1\} & \forall (u,v) \in A \label{eq:af:x} \\ 
& \alpha_v \leq  b_v \leq \beta_v & \forall v \in V \label{eq:af:time:b} \\ 
& p_{uv}  \in  \{0,1\} & \forall  \{u,v\}\in D, u<v  \label{eq:af:td:p} \\ 
& q_v\le l_{v}  \le \VehCap & \forall v \in V \label{eq:af:load:l}
\end{align}
\end{subequations}
The objective~\eqref{eq:af:obj} minimizes the total travel cost. Constraints~\eqref{eq:af:vehicle:limit} enforce that at most $|\Veh|$ vehicles depart from the depot, while equations \eqref{eq:af:ass} ensure that exactly one vehicle travels to each task. Flow conservation constraints~\eqref{eq:af:flow} guarantee that any vehicle that arrives at the location of a task must also depart from it. Inequalities~\eqref{eq:af:time:mtz} make sure that a vehicle can only start a task after completing its preceding task and consecutive travel.  Inequalities~\eqref{eq:af:td:uv:le}-\eqref{eq:af:td:vu:ge} ensure the temporal dependencies. They are identical to inequalities \eqref{eq:ff:td:uv:le}-\eqref{eq:ff:td:vu:ge} described in \cref{ss:prob:def:frag:form}. 
Constraints \eqref{eq:af:load:mtz} enforce that the cumulative served demand of a vehicle at a tasks is at least the amount of served demand at its predecessor plus the demand of the task itself. Restrictions \eqref{eq:af:x}-\eqref{eq:af:load:l} define the domains of the decision variables. 

\section{Separation of FSECs}\label[appsec]{app:sep:FSEC:alg}
This section describes the separation of FSECs based on enumerating all subsets of tasks with temporal dependencies $S \subseteq \TDVis$ whose cardinality is at most $k_\mathrm{max}$ and which uses the support graph $G[\xsol]= (\{0\} \cup \TDVis,A[\xsol])$ implied from the current solution $\xsol$, i.e., $A[\xsol] = \{(u,v) : u,v \in \{0\} \cup \TDVis \land \exists F \in \CGfragFromTo{u}{v} : \xsol_{F} > 0\}$ with arc weights $w_{uv} = \sum_{F \in \CGfragFromTo{u}{v}} \xsol_{F}$, $\forall (u,v) \in A[\xsol]$. 

For each subset $S \subseteq \TDVis$ with $|S| \leq k_\mathrm{max}$, if $\sum_{(u,v) \in A[\xsol]: u,v \in S} w_{uv} > 1$, the separation procedure checks if the corresponding FSEC is violated by computing the minimum number of vehicles $V_\mathrm{min}(S)$ needed to perform all tasks in $S$.
To compute $\FSECVehMin{S}$, we increase the number of required vehicles $k$ from 1 to $|S|$ until a feasible solution is found. For a given number of vehicles $k$, we enumerate all possible combinations of $k$ routes that together perform all tasks in $S$.
For each candidate set of routes, we examine its feasibility by the following steps. For notational convenience, let $D' = \{\{u,v\} \in D: u,v \in S\}$ denote the subset of temporal dependency pairs restricted to tasks in $S$. 
\begin{enumerate}
	\item Check whether each route satisfies the capacity restriction and whether all tasks can be performed consecutively within their given time windows. 
	\item For each route, calculate the earliest arrival time, earliest starting time, latest starting time, and latest departure time for each of the tasks. The values $\{\eaT{v}\}_{v \in S}, \{\esT{v}\}_{v \in S}, \{\lsT{v}\}_{v \in S},$ and $\{\ldT{v}\}_{v \in S}$ describe the scheduling flexibility of the tasks within their routes without considering temporal dependencies. 
    \item Initialize the set of currently considered temporal dependencies as $D'_{\mathrm{cons}}= \emptyset$ 
	\item Select a temporal dependency pair $\{u,v\}\in D' \setminus D'_{\mathrm{cons}}$ 
    for which the starting order of $u$ and $v$ is fixed (e.g., $u$ must start no later than $v$). Add this pair to $D'_{\mathrm{cons}}$ and initialize $D'_{\mathrm{check}} = \{\{u,v\}\}$ as the set of temporal dependencies that must be enforced. \label{alg:Vmin:new:pair:Dcons} 
    \item Select a pair $\{u,v\} \in D'_{\mathrm{check}}$ 
    whose temporal dependency must be enforced, and remove it from $D'_{\mathrm{check}}$.\label{alg:Vmin:new:pair:Dcheck} 
	\begin{enumerate}
		\item Update the earliest and latest starting times of $u$ and $v$ by removing starting times that cannot satisfy their temporal restrictions using the time window tightening procedure described in \cref{app:pre:procc}. If, for example, $u$ precedes $v$, $\esT{u} = \max\{\esT{u},\esT{v} - \dmax{uv}\}, \; \esT{v} = \max\{ \esT{v}, \esT{u} + \dmin{uv}\}, \; \lsT{u} = \min\{\lsT{u},  \lsT{v} - \dmin{uv}\}, \; \lsT{v}= \min\{\lsT{v}, \lsT{u} + \dmax{uv}\}$. 
		\item Update the timing variables of the remaining tasks in the routes containing $u$ and $v$.   
		\item If any previously enforced temporal dependency in $D'_{\mathrm{cons}}$ becomes violated, add it to $D'_{\mathrm{check}}$. 
	\end{enumerate}
		\item If $|D'_{\mathrm{check}}| > 0$, return to Step \ref{alg:Vmin:new:pair:Dcheck}. Otherwise, if $|D' \setminus D'_{\mathrm{cons}}| > 0$, return to Step \ref{alg:Vmin:new:pair:Dcons}.  
\end{enumerate}
If at any point the earliest starting time of a task exceeds its latest starting time, the current set of routes is infeasible and the algorithm terminates. Otherwise, a feasible schedule exists and can be obtained by starting each task at its updated earliest starting time. 
Notably, the procedure assumes that the starting order of each temporal dependency is either explicitly or implicitly fixed. For temporal dependencies without a fixed order, we explore the feasibility of both starting orders. 

\section{Additional results}\label[appsec]{app:ins:FSA}

This section provides additional insights into the impact of various components of \solalg. 

\paragraph{Impact of valid inequalities} 
\cref{fig:conf:gap:time} shows optimality gaps and solution times for various configurations of \solalg on instances with $50$ tasks. As the distributions for \solalg, \solalgNoBC, \solalgNoBCNoCut{TDIFI}, and \solalgNoBCNoCut{RCC} are very similar, only the distributions of \solalgNoBC are visualized. These results illustrate that adding the proposed valid inequalities increases the number of instances solved to optimality, reduces the optimality gaps, and accelerates the algorithm. 

\begin{figure}
	\FIGURE
	{\includegraphics[width=1.0\textwidth]{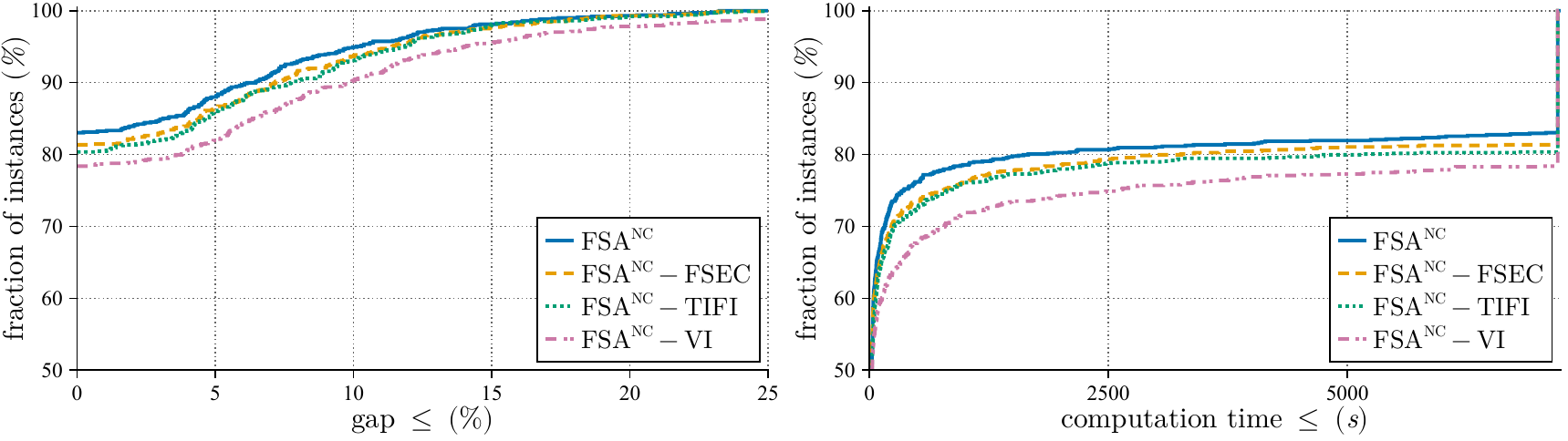}} 
	{Optimality gaps and runtimes for different configurations of \solalg on the instances with $50$ tasks. Optimality gaps are cutoff at 25\%\label{fig:conf:gap:time}}
	{}
\end{figure}

\paragraph{Column-and-row generation}
\cref{tab:FSANoBC:char:gen:lb:50,tab:FSANoBC:char:gen:lb:75} provide additional insights into the column-and-row generation (Step~\ref{sol:proc:step:lb} of our algorithm) for instances with $50$ and $75$ tasks, respectively. The tables compare various characteristics related to the inclusion of FSECs (FS), TIFIs (TI), TDIFIs (TD), and RCCs (RC), which are added in this order. They report the percentage increase of the lower bound (incr) per family of valid inequalities, the number of added valid inequalities (\#), the total separation time in seconds (time$_\mathrm{sep}$), and the computation time (time) of the column-and-row generation up to and including the considered family of valid inequalities. 
For the latter, column noRG shows the times of the column-and-row generation without separating valid inequalities. \cref{tab:FSANoBC:char:gen:lb:50,tab:FSANoBC:char:gen:lb:75} also distinguish between the subsets of instances with different numbers of temporal dependencies ($\reltempdep \in \{0.05, 0.15, 0.25\}$) and tight (type 1) and looser (type 2) time windows and capacity restrictions. 

\begin{table}
	\TABLE
	{Performance characteristics of the column-and-row generation of \solalgNoBC on instances with 50 tasks. \label{tab:FSANoBC:char:gen:lb:50}}
	{\begin{tabular}{rrccccccccccccccccc}
			\toprule
			\multirow{2}{*}{\reltempdep} & \multirow{2}{*}{Type} & \multicolumn{4}{c}{incr} &  \multicolumn{4}{c}{$\#$} & \multicolumn{4}{c}{time$_\mathrm{sep}$} & \multicolumn{5}{c}{time} \\ 
			\cmidrule(lr){3-6} \cmidrule(lr){7-10} \cmidrule(lr){11-14} \cmidrule(lr){15-19}
			& & FS & TI & TD & RC & FS & TI & TD & RC & FS & TI & TD & RC & noRG & FS & TI & TD & RC \\ \midrule
			0.05 & 1 & 0.2 & 0.9 & 0.2 & 0.1 & 0.3 & 10.5 & 2.9 & 1.8 & 0.2 & 0.0 & 0.1 & 0.1 & 7 & 8 & 10 & 11 & 13 \\
			0.05 & 2 & 0.5 & 0.7 & 0.1 & 0.0 & 0.6 & 7.5 & 2.0 & 0.2 & 0.3 & 0.1 & 0.1 & 0.0 & 333 & 343 & 409 & 432 & 434 \\
			\cmidrule(lr){1-2} \cmidrule(lr){3-6} \cmidrule(lr){7-10} \cmidrule(lr){11-14} \cmidrule(lr){15-19}
			0.15 & 1  & 1.1 & 1.8 & 0.4 & 0.3 & 3.3 & 19.8 & 7.9 & 8.5 & 4.4 & 0.1 & 0.1 & 0.2 & 3 & 5 & 9 & 11 & 16 \\
			0.15 & 2 & 1.6 & 1.0 & 0.2 & 0.0 & 3.6 & 8.0 & 3.7 & 0.3 & 5.1 & 0.1 & 0.1 & 0.0 & 23 & 32 & 43 & 52 & 53 \\
			\cmidrule(lr){1-2} \cmidrule(lr){3-6} \cmidrule(lr){7-10} \cmidrule(lr){11-14} \cmidrule(lr){15-19}
			0.25 & 1 & 2.3 & 2.1 & 0.5 & 0.4 & 8.1 & 21.2 & 11.0 & 14.9 & 27.3 & 0.0 & 0.1 & 0.1 & 2 & 11 & 21 & 29 & 35 \\
			0.25 & 2 & 3.3 & 1.3 & 0.3 & 0.0 & 6.1 & 9.1 & 5.4 & 0.5 & 29.4 & 0.0 & 0.1 & 0.0 & 8 & 34 & 44 & 53 & 55 \\
			\cmidrule(lr){1-2} \cmidrule(lr){3-6} \cmidrule(lr){7-10} \cmidrule(lr){11-14} \cmidrule(lr){15-19}
			all & all & 1.5 & 1.3 & 0.3 & 0.2 & 3.7 & 12.9 & 5.6 & 4.5 & 11.1 & 0.0 & 0.1 & 0.1 & 61 & 70 & 87 & 95 & 98 \\
			\bottomrule
	\end{tabular}}{}
\end{table}

\begin{table}
	\TABLE
	{Performance characteristics of the column-and-row generation of \solalgNoBC on the instances with 75 tasks. \label{tab:FSANoBC:char:gen:lb:75}}
	{\begin{tabular}{rrccccccccccccccccc}
			\toprule
			\multirow{2}{*}{\reltempdep} & \multirow{2}{*}{Type} & \multicolumn{4}{c}{incr} &  \multicolumn{4}{c}{$\#$} & \multicolumn{4}{c}{time$_\mathrm{sep}$} & \multicolumn{5}{c}{time} \\ 
			\cmidrule(lr){3-6} \cmidrule(lr){7-10} \cmidrule(lr){11-14} \cmidrule(lr){15-19}
			& & FS & TI & TD & RC & FS & TI & TD & RC & FS & TI & TD & RC & noRG & FS & TI & TD & RC \\ \midrule
			0.05 & 1  & 0.1 & 1.0 & 0.2 & 0.8 & 0.5 & 25.5 & 7.4 & 25.4 & 0.4 & 0.1 & 0.1 & 0.2 & 37 & 38 & 49 & 57 & 71 \\
			0.05 & 2 & 0.2 & 0.6 & 0.2 & 0.0 & 1.3 & 14.4 & 4.8 & 0.4 & 0.7 & 0.1 & 0.1 & 0.0 & 709 & 790 & 1024 & 1139 & 1149\\
			\cmidrule(lr){1-2} \cmidrule(lr){3-6} \cmidrule(lr){7-10} \cmidrule(lr){11-14} \cmidrule(lr){15-19}
			0.15 & 1 & 0.9 & 2.2 & 0.5 & 0.9 & 6.0 & 46.3 & 16.8 & 48.3 & 30.3 & 0.1 & 0.2 & 0.4 & 8 & 14 & 32 & 45 & 59\\
			0.15 & 2 & 1.1 & 1.3 & 0.3 & 0.0 & 7.1 & 20.5 & 8.8 & 1.7 & 23.9 & 0.1 & 0.2 & 0.0 & 56 & 93 & 158 & 200 & 210\\
			\cmidrule(lr){1-2} \cmidrule(lr){3-6} \cmidrule(lr){7-10} \cmidrule(lr){11-14} \cmidrule(lr){15-19}
			0.25 & 1 & 1.6 & 2.6 & 0.7 & 1.1 & 12.9 & 47.0 & 23.8 & 48.8 & 168.0 & 0.1 & 0.1 & 0.4 & 6 & 33 & 93 & 141 & 193  \\
			0.25 & 2 & 2.0 & 1.5 & 0.4 & 0.1 & 11.6 & 20.3 & 10.5 & 1.5 & 122.5 & 0.1 & 0.1 & 0.0 & 30 & 113 & 178 & 226 & 237  \\
			\cmidrule(lr){1-2} \cmidrule(lr){3-6} \cmidrule(lr){7-10} \cmidrule(lr){11-14} \cmidrule(lr){15-19}
			all & all& 1.0 & 1.6 & 0.4 & 0.5 & 6.6 & 29.4 & 12.2 & 21.7 & 58.0 & 0.1 & 0.1 & 0.2 & 137 & 175 & 248 & 293 & 312  \\
			\bottomrule
	\end{tabular}}{}
\end{table}

\cref{tab:FSANoBC:char:gen:lb:50} shows that the separation and addition of valid inequalities increases the lower bound by 3.3\%, with FSECs and TIFIs having the largest impact (1.5\% and 1.3\%). The results, thereby, underestimate the impact of FSECs, since a subset of FSECs is already included at the start of \solalg (see \cref{sec:compres}). On average, the contribution of valid inequalities to the lower bound increases with the number of temporal dependencies.
FSECs have the greatest impact on instances with looser time window and capacity restrictions, while TIFIs, TDIFIs, and RCCs have a greater impact for more restrictive instances. As expected, RCCs do not improve the lower bound when capacity constraints are non restrictive, i.e., for instances of type 2.  
The number of added valid inequalities increases with the number of temporal dependencies, which is in line with their impact on the lower bound. More TIFIs, TDIFIs, and RCCs are added for the more restrictive type 1 instances, but the average number of added inequalities remains limited. The separation time of FSECs is on average almost 30 seconds for the instances with the largest number of temporal dependencies, while the separation time of the remaining sets of valid inequalities is negligible. Adding valid inequalities increases the average computation time of the column-and-row generation from 61 to 98 seconds. The reported separation time of the FSECs thereby exceeds the difference in computation time between noRG and FS, as it represents the cumulative separation time during the lower bound computation. 

Similar observations can be made for the results on instances with $75$ tasks, i.e., from \cref{tab:FSANoBC:char:gen:lb:75}. Compared to the results on the smaller instances with 50 tasks, the impact of the TIFIs and RCCs on the lower bound increases, while that of the FSECs decreases. Moreover, the number of added valid inequalities increases for all types. Separating FSECs consumes significantly more time, particularly for instances with the largest number of dependencies. The total computation time of the column-and-row generation increases from 98 to 312 seconds.

\paragraph{Enumeration and branch-and-cut}

\cref{tab:FSANoBC:char:enum:phase:50,tab:FSANoBC:char:enum:phase:75} summarize characteristics of the enumeration and branch-and-cut procedures for instances with 50 and 75 tasks. They report the average number of enumeration iterations (steps), the number of enumerated fragments (\#$_{\mathrm{frag}}$), the number of instances for which the number of enumerated fragments exceeds the maximum allowed number of fragments (\#$_{\maxfragenum}$), the fraction of instances for which the final primal bound is identified in Step~\ref{sol:proc:step:init:ub} (frac$_{\mathrm{P\ref{sol:proc:step:init:ub}}}$), the computation time in seconds (time), and the percentage of removed fragments during the last execution of the reduction techniques (rem). Moreover, TOT denotes all the procedures of Steps~\ref{sol:proc:step:init:ub}-\ref{sol:proc:recursion}, EA the enumeration algorithm, RB the fragment reduction technique based on reduced costs bounds, RD the fragment reduction techniques based on obtaining new dual costs, and BC the branch-and-cut (without adding cutting planes in \solalgNoBC).

\begin{table}
	\TABLE
	{Performance characteristics of Steps~\ref{sol:proc:step:init:ub}-\ref{sol:proc:recursion} of \solalgNoBC on instances with 50 tasks. \label{tab:FSANoBC:char:enum:phase:50}}
	{\begin{tabular}{rrccccccccccccc}
			\toprule  
			\multirow{2}{*}{\reltempdep} & \multirow{2}{*}{Type} & \multirow{2}{*}{steps} & \multirow{2}{*}{\#$_{\mathrm{frag}}$} & \multirow{2}{*}{frac$_{\mathrm{P\ref{sol:proc:step:init:ub}}}$} & \multirow{2}{*}{\#$_{\maxfragenum}$}  & \multicolumn{5}{c}{time} &  \multicolumn{2}{c}{rem} \\ 
			\cmidrule(lr){7-11} \cmidrule(lr){12-13}
			& & & & & & TOT & EA & RB & RD & BC & RB & RD  \\ \midrule
			0.05 & 1 & 1.5 & 80744 & 0.4 & 0 & 330 & 2 & 1 & 2 & 325 & 1.7 & 25.0  \\
			0.05 & 2 & 1.0 & 4516874 & 0.7 & 28 & 1260 & 894 & 3 & 35 & 329 & 4.8 & 26.5 \\ 
			\cmidrule(lr){1-2} \cmidrule(lr){3-6} \cmidrule(lr){7-11} \cmidrule(lr){12-13} 
			0.15 & 1 & 1.5 & 218504 & 0.5 & 0 & 1059 & 12 & 3 & 9 & 1036 & 3.6 & 25.5  \\
			0.15 & 2 & 1.1 & 1741423 & 0.5 & 4 & 1631 & 128 & 9 & 41 & 1453 & 9.2 & 31.7 \\
			\cmidrule(lr){1-2} \cmidrule(lr){3-6} \cmidrule(lr){7-11} \cmidrule(lr){12-13} 
			0.25 & 1 & 1.4 & 78224 & 0.5 & 0 & 1209 & 3 & 3 & 3 & 1200 & 4.4 & 23.8  \\
			0.25 & 2 & 1.1 & 658479 & 0.5 & 2 & 1677 & 27 & 11 & 11 & 1628 & 8.2 & 33.3 \\
			\cmidrule(lr){1-2} \cmidrule(lr){3-6} \cmidrule(lr){7-11} \cmidrule(lr){12-13} 
			all & all & 1.3 & 1176784 & 0.5 & 34 & 1183 & 172 & 5 & 16 & 990 & 5.3 & 27.5  \\			
			\bottomrule
	\end{tabular}}{}
\end{table}

\begin{table}
	\TABLE
	{Performance characteristics of Steps~\ref{sol:proc:step:init:ub}-\ref{sol:proc:recursion} of \solalgNoBC on instances with 75 tasks. \label{tab:FSANoBC:char:enum:phase:75}}
	{\begin{tabular}{rrccccccccccccc}
			\toprule  
			\multirow{2}{*}{\reltempdep} & \multirow{2}{*}{Type} & \multirow{2}{*}{steps} & \multirow{2}{*}{\#$_{\mathrm{frag}}$} & \multirow{2}{*}{frac$_{\mathrm{P\ref{sol:proc:step:init:ub}}}$} & \multirow{2}{*}{\#$_{\maxfragenum}$}  & \multicolumn{5}{c}{time} &  \multicolumn{2}{c}{rem} \\ 
			\cmidrule(lr){7-11} \cmidrule(lr){12-13}
			& & & & & & TOT & EA & RB & RD & BC & RB & RD  \\ \midrule 
			0.05 & 1 & 1.0 & 3680371 & 0.3 & 13 & 1504 & 189 & 1 & 73 & 1241 & 1.0 & 26.5 \\
			0.05 & 2 & 1.0 & 9779456 & 0.8 & 69 & 1887 & 1693 & 1 & 38 & 156 & 3.0 & 20.1 \\ 
			\cmidrule(lr){1-2} \cmidrule(lr){3-6} \cmidrule(lr){7-11} \cmidrule(lr){12-13}
			0.15 & 1  & 1.1 & 2739625 & 0.3 & 14 & 2158 & 71 & 1 & 45 & 2041 & 3.5 & 25.9   \\
			0.15 & 2 & 1.0 & 9963935 & 0.6 & 69 & 1040 & 458 & 2 & 43 & 537 & 4.9 & 25.4 \\
			\cmidrule(lr){1-2} \cmidrule(lr){3-6} \cmidrule(lr){7-11} \cmidrule(lr){12-13}
			0.25 & 1 & 1.1 & 1748731 & 0.3 & 8 & 2495 & 41 & 1 & 41 & 2412 & 4.2 & 23.6 \\
			0.25 & 2 & 1.1 & 6544145 & 0.5 & 38 & 1619 & 239 & 4 & 62 & 1315 & 5.8 & 27.6 \\
			\cmidrule(lr){1-2} \cmidrule(lr){3-6} \cmidrule(lr){7-11} \cmidrule(lr){12-13} 
			all & all & 1.1 & 5634861 & 0.4 & 211 & 1794 & 436 & 1 & 50 & 1306 & 3.7 & 24.9 \\
			\bottomrule
	\end{tabular}}{}
\end{table}

\cref{tab:FSANoBC:char:enum:phase:50} shows that \solalgNoBC requires, on average, 1.3 enumeration steps, during which 1\,176\,784 fragments are generated. The number of identified fragments is higher for type 2 instances with less restrictive time window and capacity restrictions. Interestingly, the initial upper bound obtained during Step~\ref{sol:proc:step:init:ub} is equal to the final primal bound for half of the instances, demonstrating the effectiveness of this branch-and-cut procedure. 
Moreover, in 34 out of 1008 instances, all of type 2, the enumeration procedure reaches the maximum number of fragments, after which the solution algorithm terminates. 
On average, \solalgNoBC spends the most time in the branch-and-cut procedures, which takes 990 seconds out of the total of 1183 seconds. Both the total computation time and share spent on branch-and-cut increase with the number of temporal dependencies. Only for instances with $\reltempdep=0.05$, the enumeration algorithm contributes most to the computation time. Fragment reduction techniques require, on average, only 21 seconds, while reducing the number of fragments by nearly 33\%. The fragment reduction technique based on resolving the linear program is the most effective, eliminating 27.5\% of the fragments. A larger proportion of fragments is removed for type 2 instances. 

The same trends are observed for instances with 75 tasks, as summarized in \cref{tab:FSANoBC:char:enum:phase:75}. Compared to the results on instances with 50 tasks, the average number of enumeration steps decreases from 1.3 to 1.1, while the number of enumerated fragments increases from 1\,176\,784 to 5\,634\,861. 
Notably, \solalgNoBC terminates early for roughly 20\% of the instances, compared to 3\% for instances with 50 tasks, because the enumeration procedure identifies more than $\maxfragenum$ fragments.  
The total computation time of Steps~\ref{sol:proc:step:init:ub}-\ref{sol:proc:recursion} increases from 1183 to 1794 seconds, with particularly large increases for instances with tight time window and capacity restrictions. The enumeration procedure also consumes significantly more time than for instances with 50 tasks, whereas fragment reduction techniques remain relatively fast and still remove nearly 30\% of the identified fragments.

\end{APPENDICES}








  



\end{document}